\documentclass[reqno]{amsart}
\usepackage{amsmath}
\usepackage{amsthm, amscd, amssymb, amsfonts, amsbsy}
\usepackage[usenames, dvipsnames]{color}
\usepackage{enumerate}
\usepackage{mathrsfs}
\usepackage{verbatim}

\numberwithin{equation}{section}

\theoremstyle{plain}
\newtheorem{theorem}{Theorem}[section]
\newtheorem{lemma}[theorem]{Lemma}
\newtheorem{corollary}[theorem]{Corollary}
\newtheorem{proposition}[theorem]{Proposition}

\theoremstyle{definition}
\newtheorem{definition}[theorem]{Definition}
\newtheorem{assumption}[theorem]{Assumption}

\theoremstyle{remark}
\newtheorem{remark}[theorem]{Remark}
\newtheorem{notation}{Notation}[section]

\makeatletter
\def\dashint{\operatorname%
{\,\,\text{\bf--}\kern-.98em\DOTSI\intop\ilimits@\!\!}}
\makeatother

\def\bR{\mathbb{R}}
\def\cA{\mathcal{A}}

\def\cK{\mathcal{K}}
\def\cL{\mathcal{L}}
\def\cG{\mathcal{G}}
\def\cP{\mathcal{P}}
\def\cF{\mathcal{F}}
\def\cH{\mathcal{H}}

\begin{document}
\title[Stokes system]{Green functions for pressure of Stokes systems}

\author[J. Choi]{Jongkeun Choi}
\address[J. Choi]{School of Mathematics, Korea Institute for Advanced Study, 85 Hoegiro, Dongdaemun-gu, Seoul 02455, Republic of Korea}
\email{jkchoi@kias.re.kr}

\thanks{}

\author[H. Dong]{Hongjie Dong}
\address[H. Dong]{Division of Applied Mathematics, Brown University, 182 George Street, Providence, RI 02912, USA}

\email{Hongjie\_Dong@brown.edu}

\thanks{H. Dong was partially supported by the NSF under agreement DMS-1600593.}


\begin{abstract}
We study Green functions for the pressure of stationary Stokes systems in a (possibly unbounded) domain $\Omega\subset \mathbb{R}^d$, where $d\ge 2$.
We construct the Green function when coefficients are merely measurable in one direction and have Dini mean oscillation in the other directions, and $\Omega$ is such that the divergence equation is solvable there.
We also establish global pointwise bounds for the Green function and its derivatives when coefficients have Dini mean oscillation and $\Omega$ has a $C^{1,\rm{Dini}}$ boundary.
Green functions for the flow velocity of Stokes systems are also considered.
\end{abstract}

\maketitle

%========================================
\section{Introduction}			\label{S1}
%========================================

We study Green functions and fundamental solutions for stationary Stokes systems with variable coefficients.
Let $\cL$ be a second order elliptic operator in divergence form
$$
\cL u=D_\alpha (A^{\alpha\beta}D_\beta u)
$$
acting on column vector valued functions $u=(u_1,\ldots,u_d)^\top$ defined on a domain $\Omega\subset \bR^d$,
where $d\ge 2$.
Unlike elliptic systems, Stokes systems have
two types of Green functions.
One is a pair $(G, \Pi)=(G(x,y), \Pi(x,y))$, which we call  {\em{Green function for the flow velocity}}, satisfying
\begin{equation}		\label{180323@eq1}
\left\{
\begin{aligned}
\cL G(\cdot,y)+\nabla \Pi(\cdot,y)=\delta_y I &\quad \text{in }\, \Omega,\\
\operatorname{div} G(\cdot,y)=0 &\quad \text{in }\, \Omega,\\
G(\cdot,y)=0 &\quad \text{on }\, \partial \Omega.
\end{aligned}
\right.
\end{equation}
Here, $G$ is a $d\times d$ matrix-valued function, $\Pi$ is a $d\times 1$ vector-valued function, $I$ is the $d\times d$ identity matrix, and $\delta_y$ is the Dirac delta function concentrated at $y$.
For a more precise definition of the Green function for the flow velocity, see Section \ref{S7}.
The other one is a pair $(\cG, \cP)=(\cG(x,y), \cP(x,y))$, which we call  {\em{Green function for the pressure}}, satisfying (for instance, when $|\Omega|<\infty$)
\begin{equation}		\label{180323@eq2}
\left\{
\begin{aligned}
\cL \cG(\cdot,y)+\nabla \cP(\cdot,y)=0 &\quad \text{in }\, \Omega\setminus \{y\},\\
\operatorname{div} \cG(\cdot,y)=\delta_y-\frac{1}{|\Omega|} &\quad \text{in }\, \Omega,\\
\cG(\cdot,y)=0 &\quad \text{on }\, \partial \Omega.
\end{aligned}
\right.
\end{equation}
Here, $\cG$ is a $d\times 1$ vector-valued function and $\cP$ is a real-valued function.
For a more precise definition of the Green function for the pressure, see Section \ref{S2-3}.
An observation is that if there exist Green functions for the flow velocity and the pressure, then the pair $(u, p)$ given by
$$
u(y)=\int_\Omega G(x,y)^{\top}f(x)\,dx, \quad p(y)=-\int_\Omega \cG(x,y)\cdot f(x)\,dx
$$
is a weak solution of the problem
$$
\left\{
\begin{aligned}
\cL^* u+\nabla p=f &\quad \text{in }\, \Omega,\\
\operatorname{div} u=0 &\quad \text{in }\, \Omega,\\
u=0 &\quad \text{on }\, \partial \Omega,
\end{aligned}
\right.
$$
where $\cL^*$ is the adjoint operator of $\cL$.

There
is a large body of literature concerning Green function
for the flow velocity satisfying \eqref{180323@eq1}.
Regarding the classical Stokes system with the Laplace operator $\cL  =\Delta$, i.e.,
$$
A^{\alpha\beta}=[\delta_{\alpha\beta}\delta_{ij}]_{i,j=1}^d \quad \text{($\delta_{ij}$ is the usual Kronecker delta symbol)},
$$
we refer the reader to Ladyzhenskaya \cite{MR0254401}, Maz'ya-Plamenevski{\u\i} \cite{MR0725151, MR0734895}, and D. Mitrea-I. Mitrea \cite{MR2763343}.
In \cite{MR0254401}, the author provided explicit formulas of fundamental solutions in $\bR^2$ and $\bR^3$.
In \cite{MR0725151, MR0734895}, the authors established pointwise estimates for the Green function and its derivatives in a piecewise smooth domain in $\bR^3$.
In \cite{MR2763343}, the authors proved the existence of the Green function in a Lipschitz domain in $\bR^d$, where $d\ge 2$.
For further related results on fundamental solutions and Green functions, one can refer to the book \cite{MR2808162} and references therein.
See also \cite{MR2182091, MR3320459} for Green functions satisfying mixed boundary conditions in domains in $\bR^3$ and $\bR^2$.
Regarding Stokes systems with variable coefficients, i.e.,
$$
A^{\alpha\beta}=[A^{\alpha\beta}_{ij}(x)]_{i,j=1}^d,
$$
we refer the reader to \cite{MR3693868, MR3670039, MR3906316}.
In \cite{MR3693868}, the authors established the existence and pointwise bound of the Green function on a bounded $C^1$ domain  when $d\ge 3$ and $A^{\alpha\beta}$ have vanishing mean oscillations (VMO).
The corresponding results were obtained in \cite{MR3670039} in the whole space and a half space when $A^{\alpha\beta}$ are merely measurable in one direction and have small mean oscillations in the other directions (partially BMO).
In \cite{MR3906316}, the authors constructed the Green function in a two dimensional domain when $A^{\alpha\beta}$ are measurable and bounded.
They also considered pointwise bounds of the Green function and its derivatives.
For further related results, one can refer to
 \cite{arXiv:1710.05383} for the Green function of the Stokes system with oscillating periodic coefficients and \cite{MR3877495} for the Green function satisfying the conormal derivative boundary condition.

In contrast to Green functions for the flow velocity,
there are relatively few results on Green functions for the pressure satisfying \eqref{180323@eq2}.
In particular, we are not able to find any literature dealing with Green functions for the pressure of the Stokes system with variable coefficients.
Restricted to the Stokes system with the Laplace operator, we refer the reader to \cite{MR0734895}, where the authors proved the pointwise estimate for the Green function (and its derivatives) of the Dirichlet problem in a three dimensional domain.
The corresponding results for the mixed problem were obtained in \cite{MR2182091}.

In this paper, we are concerned with both Green functions and fundamental solutions for the pressure of  Stokes systems with variable coefficients.
The class of coefficients $A^{\alpha\beta}$ we consider is called of {\em{partially Dini mean oscillation}}, which means that $A^{\alpha\beta}$ are merely measurable in one direction and have Dini mean oscillation in the other directions; see Definition \ref{D1}.
Stokes systems with irregular coefficients of this type  may be used to describe the motion of inhomogeneous fluids with density dependent viscosity and two or multiple fluids with interfacial boundaries; see \cite{MR0425391,MR1422251,MR2663713,MR3758532}.

Let $\Omega$ be a (possibly unbounded) domain in $\bR^d$ satisfying an exterior measure condition when $d=2$, and assume that the divergence equation is solvable in $\Omega$.
We prove that if coefficients are of partially Dini mean oscillation, then there exists a unique Green function $(\cG, \cP)$ for the pressure in $\Omega$.
The Green function satisfies global pointwise bounds
\begin{equation}		\label{180328@eq1}
\begin{aligned}
|\cG(x,y)|\le C|x-y|^{1-d},  &\quad x\neq y,\\
|D_x \cG(x,y)|+|\cP(x,y)|\le C|x-y|^{-d}, &\quad x\neq y
\end{aligned}
\end{equation}
if we assume further that coefficients are of Dini mean oscillation in the {\em{all}} directions and $\Omega$ has a $C^{1,\rm{Dini}}$ boundary.
Especially, the fundamental solution $(d\ge 3)$ and the Green function $(d\ge 2)$ in a half space  have the global pointwise bounds \eqref{180328@eq1} under a weaker condition that coefficients are of partially Dini mean oscillation.
For further details, see Section \ref{S3}.

We also deal with the Green function $(G,\Pi)$ (and  the fundamental solution) for the flow velocity of Stokes system.
As mentioned above, its existence and pointwise bound, i.e.,
$$
|G(x,y)|\le C|x-y|^{2-d}, \quad x\neq y,
$$
were obtained in \cite{MR3693868, MR3670039}  when $d\ge 3$.
In this paper, we extend the results in \cite{MR3693868, MR3670039} by showing that
$$
|D_xG(x,y)|+|\Pi(x,y)|\le C|x-y|^{1-d}, \quad x\neq y,
$$
under the stronger assumption that the coefficients are of Dini mean oscillation in a domain having a $C^{1,\rm{Dini}}$ boundary.
Moreover, we verify a symmetric property of Green functions for the flow velocity and  the pressure.
For further details, see Section \ref{S7}.

The theory regarding the existence and estimates of Green functions for Stokes systems is closely related to regularity theory of solutions to
\begin{equation}		\label{180329@eq1}
\left\{
\begin{aligned}
\cL u+\nabla p=f &\quad \text{in }\, \Omega,\\
\operatorname{div} u=g &\quad \text{in }\, \Omega.
\end{aligned}
\right.
\end{equation}
When dealing with Green functions for the flow velocity in \cite{ MR3693868, MR3670039, MR3877495}, the authors used $L^\infty$ or $C^{\alpha}$-estimates of solutions $u$, which can be obtained from $W^{1,q}$-estimates for the system \eqref{180329@eq1}.
See \cite{MR3693868,arXiv:1604.02690, MR3758532,MR3809039} for $W^{1,q}$-regularity results with $q\in (1,\infty)$.
This approach was introduced by Hofmann-Kim \cite{MR2341783} to deal with Green functions and fundamental solutions for elliptic systems with VMO coefficients.
In this paper, to construct the Green function for the pressure, we utilize $L^\infty$-estimates of not only $u$ but also $(Du,p)$.
Thus, we are not able to apply the aforementioned $W^{1,q}$-estimates.
Instead, we employ the recent results in \cite{MR3912724, MR3890946}, where we proved $W^{1,\infty}$ and $C^1$-estimates for Stokes systems with coefficients having (partially) Dini mean oscillation.
This argument allows us to get pointwise bounds of the Green function as well as its derivatives.

The remainder of this paper is organized as follows.
We introduce some notation and definitions in the next section.
In Section \ref{S3}, we state the main theorems regarding Green functions for the pressure.
In Section \ref{S_Pre}, we present some preliminary results, and in Sections \ref{S_Ap} and \ref{S_Pr}, we provide the proofs of the main theorems.
We devote Section \ref{S7} to Green functions for the flow velocity.
In Appendix, we provide the proofs of $L^\infty$-estimates, which are crucial for proving our main theorems.

%========================================
\section{Preliminaries}			\label{S2}
%========================================

In this section, we introduce some notation and definitions used throughout the paper.

%========================================
\subsection{Notation}	
%========================================
We use $x=(x_1,x')$ to denote a point in the Euclidean space $\bR^d$, where $d\ge 2$ and $x'=(x_2,\ldots,x_d)\in \bR^{d-1}$.
We also write $y=(y_1,y')$ and $x_0=(x_{01},x_0')$, etc.
Balls in $\bR^d$ and $\bR^{d-1}$ are defined by
$$
B_r(x)=\{y\in \bR^d:|x-y|<r\}, \quad B_r'(x')=\{y'\in \bR^{d-1}: |x'-y'|<r\}.
$$
Let $\Omega$ be a domain in $\bR^d$.
We write $\Omega_r(x)=\Omega\cap B_r(x)$ for all $x\in \bR^d$ and $r>0$.

For $q\in (0,\infty]$, we define
$$
\tilde{L}^q(\Omega)=
\{u\in L^q(\Omega): (u)_{\Omega}=0\},
$$
where $L^q(\Omega)$ is the set of all measurable functions on $\Omega$ that are $q$-th integrable, and $(u)_\Omega$ is the average of $u$ over $\Omega$, i.e.,
$$
(u)_\Omega=
\left\{
\begin{aligned}
\dashint_\Omega u\,dx=\frac{1}{|\Omega|} \int_\Omega u\,dx &\quad \text{if }\, |\Omega|<\infty,\\
0 &\quad \text{if }\, |\Omega|=\infty.
\end{aligned}
\right.
$$
Note that $\tilde{L}^q(\Omega)=L^q(\Omega)$ if $|\Omega|=\infty$.

For $q\in [1,\infty]$, we denote by $W^{1,q}(\Omega)$ the usual Sobolev space and by $W^{1,q}_0(\Omega)$ the completion of $C^\infty_0(\Omega)$ in $W^{1,q}(\Omega)$, where $C^\infty_0(\Omega)$ is the set of all infinitely differentiable functions with compact supports in $\Omega$.

For $q\in [1,d)$, the space $Y^{1,q}(\Omega)$ is defined as the set of all measurable functions $u$ on $\Omega$ having a finite norm
$$
\|u\|_{Y^{1,q}(\Omega)}=\|u\|_{L^{dq/(d-q)}(\Omega)}+\|Du\|_{L^q(\Omega)},
$$
and the space $Y^{1,q}_0(\Omega)$ is defined as the completion of $C^\infty_0(\Omega)$ in $Y^{1,q}(\Omega)$.
Note that
$$
Y^{1,q}_0(\bR^d)=Y^{1,q}(\bR^d), \quad W^{1,q}_0(\Omega)\subset Y^{1,q}_0(\Omega),
$$
$$
W^{1,q}_0(\Omega)=Y^{1,q}_0(\Omega) \quad \text{if }\, |\Omega|<\infty,
$$
and $Y^{1,2}_0(\Omega)$ $(d\ge 3)$ is a Hilbert space with inner product
\begin{equation}		\label{180726@eq2}
\langle u, v\rangle =\int_{\Omega} Du \cdot Dv \,dx.
\end{equation}
When $d=2$, we denote by $Y^{1,2}_0(\Omega)$  the set of all weakly differentiable functions $u$ on $\Omega$ such that $\nabla u\in L^2(\Omega)$ and $\eta u\in W^{1,2}_0(\Omega)$ for any $\eta\in C^\infty_0(\bR^2)$.
In this case, if $\Omega$ is a Green domain in $\bR^2$, i.e.,
$$
\{u\chi_\Omega: u\in C^\infty_0(\bR^2)\} \not\subset W^{1,2}_0(\Omega),
$$
then $Y^{1,2}_0(\Omega)$ is also a Hilbert space with inner product \eqref{180726@eq2} and $C^\infty_0(\Omega)$ is a dense subset of $Y^{1,2}_0(\Omega)$.
We note that $\bR^2$ is not a Green domain, but $\bR^2_+$ is.
If $|\Omega|<\infty$, then $\Omega$ is a Green domain and $W^{1,2}_0(\Omega)=Y^{1,2}_0(\Omega)$.
For more discussions of the space $Y^{1,q}_0(\Omega)$, see \cite[\S 1.3.4]{MR1461542}.

We say that a measurable function $\omega: [0,a]\to [0, \infty)$ is a Dini function provided that there are constants $c_1,\, c_2>0$ such that
$$
c_1\omega(t)\le \omega(s)\le c_2\omega(t) \quad \text{whenever }\, 0<\frac{t}{2}\le s\le t\le a
$$
and that $\omega$ satisfies the Dini condition
$$
\int_0^a \frac{\omega(t)}{t}\,dt<\infty.
$$

\begin{definition}		\label{D1}
Let $f\in L^1_{\rm{loc}}(\overline{\Omega})$.
\begin{enumerate}[$(a)$]
\item
We say that $f$ is of {\em{partially Dini mean oscillation with respect to $x'$}} in the interior of $\Omega$ if there exists a Dini function $\omega:[0,1]\to [0,\infty)$ such that for any $x\in \Omega$ and $r\in (0,1]$ satisfying $B_{2r}(x)\subset \Omega$, we have
$$
\dashint_{B_r(x)}\bigg|f(y)-\dashint_{B_r'(x')}f(y_1,z')\,dz'\bigg|\,dy\le \omega(r).
$$
\item
We say that $f$ is of {\em{Dini mean oscillation}} in $\Omega$ if there exists a Dini function $\omega:[0, 1]\to [0,\infty)$ such that for any $x\in \overline{\Omega}$ and $r\in (0,1]$, we have
$$
\dashint_{\Omega_r(x)}\bigg|f(y)-\dashint_{\Omega_r(x)} f(z)\,dz \bigg|\,dy\le \omega(r).
$$
\end{enumerate}
\end{definition}

We define a domain with a $C^{1,\rm{Dini}}$ boundary by locally the graph of a $C^1$ function whose derivatives are uniformly Dini continuous.

\begin{definition}		\label{D6}
Let $\Omega$ be a  domain in $\bR^d$.
We say that $\Omega$ has a $C^{1,\rm{Dini}}$ boundary if there exist a constant $R_0\in (0,1]$ and a Dini function $\varrho_0:[0, 1]\to [0,\infty)$ such that the following holds:
For any $x_0=(x_{01},x_0')\in \partial \Omega$, there exist a $C^1$ function $\chi:\bR^{d-1}\to \bR$ and a coordinate system depending on $x_0$ such that in the new coordinate system, we have
$$
|\nabla_{x'}\chi(x_0')|=0, \quad \Omega_{R_0}(x_0)=\{x\in B_{R_0}(x_0): x_1>\chi(x')\},
$$
and
$$
\varrho_{\chi}(r)\le \varrho_0(r) \quad \text{for all }\, r\in [0,R_0],
$$
where $\varrho_{\chi}$ is the modulus of continuity of $\nabla_{x'}\chi$, i.e.,
$$
\varrho_{\chi}(r)=\sup \big\{|\nabla_{x'} \chi(y')-\nabla_{x'} \chi(z')| : y',z'\in \bR^{d-1}, \, |y'-z'|\le r \big\}.
$$
\end{definition}		

%========================================
\subsection{Stokes system}			\label{S2-2}
%========================================

Let $\cL$ be a strongly elliptic operator of the form
$$
\cL u=D_\alpha(A^{\alpha\beta}D_\beta u),
$$
where the coefficients $A^{\alpha\beta}=A^{\alpha\beta}(x)$ are $d\times d$ matrix-valued functions on $\bR^d$ satisfying the strong ellipticity condition,
i.e., there is a constant $\lambda\in (0,1]$ such that for any $x\in \bR^d$ and $\xi_\alpha\in \bR^d$, $\alpha\in \{1,\ldots,d\}$, we have
\begin{equation}		\label{180307@eq1}
|A^{\alpha\beta}(x)|\le \lambda^{-1}, \quad \sum_{\alpha,\beta=1}^d A^{\alpha\beta}(x)\xi_\beta \cdot \xi_\alpha\ge \lambda \sum_{\alpha=1}^d |\xi_\alpha|^2.
\end{equation}
The adjoint operator $\cL^*$ of $\cL$ is defined by
$$
\cL^* u=D_\alpha((A^{\beta\alpha})^{\top}D_\beta u),
$$
where $(A^{\beta \alpha})^{\top}$ is the transpose of $A^{\beta \alpha}$.
Note that the coefficients of $\cL^*$ also satisfy the ellipticity condition \eqref{180307@eq1} with the same constant $\lambda$.

Let $\Omega$ be a domain in $\bR^d$.
We say that
$$
(u, p)\in W^{1,1}_{\rm{loc}}(\Omega)^d\times L^1_{\rm{loc}}(\Omega)
$$
is a weak solution of
$$
\cL u+\nabla p=f+D_\alpha f_\alpha \quad \text{in }\, \Omega
$$
if
$$
\int_{\Omega} A^{\alpha\beta}D_\beta u\cdot D_\alpha \phi\,dx+\int_\Omega p \operatorname{div}\phi\,dx=-\int_\Omega f\cdot \phi\,dx+\int_\Omega f_\alpha \cdot D_\alpha \phi\,dx
$$
holds for any $\phi\in C^\infty_0(\Omega)^d$.
Similarly, we say that
$$
(u, p)\in W^{1,1}_{\rm{loc}}(\Omega)^d \times L^1_{\rm{loc}}(\Omega)
$$
is a weak solution of
$$
\cL^* u+\nabla p=f+D_\alpha f_\alpha \quad \text{in }\, \Omega
$$
if
$$
\int_{\Omega} A^{\alpha\beta}D_\beta \phi \cdot D_\alpha u\,dx+\int_\Omega p \operatorname{div}\phi\,dx=-\int_\Omega f\cdot \phi\,dx+\int_\Omega f_\alpha \cdot D_\alpha \phi\,dx
$$
holds for any $\phi\in C^\infty_0(\Omega)^d$.

%========================================
\subsection{Green function for the pressure}	\label{S2-3}
%========================================
In this subsection, we state the definition of a Green function for the pressure.
In the definition below, $\cG=\cG(x,y)$ is a $d\times 1$ vector-valued function and $\cP=\cP(x,y)$ is a real-valued function on $\Omega\times \Omega$.

\begin{definition}		\label{D0}
Let $d\ge 2$ and $\Omega$ be a domain in $\bR^d$.
We say that $(\cG,\cP)$ is a Green function for the pressure of $\cL$ in $\Omega$ if it satisfies the following properties.
\begin{enumerate}[$(a)$]
\item
For any $y\in \Omega$ and $R>0$,
$$
\cG(\cdot,y)\in L^1_{\rm{loc}}(\Omega)^d, \quad (1-\eta)\cG(\cdot,y)\in Y^{1,2}_0(\Omega)^d,
$$
where $\eta$ is a smooth function satisfying $\eta\equiv 1$ on $B_R(y)$.
Moreover,
$$
\cP(\cdot,y)\in L^q_{\rm{loc}}(\Omega)\cap L^{2}(\Omega\setminus \overline{B_R(y)})
$$
for any $q\in (0,1)$.
\item
For any $y\in \Omega$ and $R>0$, $(\cG(\cdot,y),\cP(\cdot,y))$ satisfies
\begin{equation}	\label{180212@A1}
\left\{
\begin{aligned}
\cL \cG(\cdot,y)+\nabla \cP(\cdot,y)=0 &\quad \text{in }\, \Omega\setminus \overline{B_R(y)},\\
\operatorname{div} \cG(\cdot,y)= \delta_y -\frac{1}{|\Omega|} &\quad \text{in }\, \Omega,\\
\cG(\cdot,y)=0 &\quad \text{on }\, \partial \Omega,
\end{aligned}
\right.
\end{equation}
where $\frac{1}{|\Omega|}=0$ if $|\Omega|=\infty$.
\item
If $(u,p)\in Y^{1,2}_0(\Omega)^d\times \tilde{L}^2(\Omega)$ is a weak solution of the problem
\begin{equation}		\label{180206@eq2}
\left\{
\begin{aligned}
\cL^* u+\nabla p=f &\quad \text{in }\, \Omega,\\
\operatorname{div} u=g-(g)_\Omega &\quad \text{in }\, \Omega,\\
u=0 &\quad \text{on }\, \partial \Omega,
\end{aligned}
\right.
\end{equation}
where
$f\in C^\infty_0(\Omega)^d$ and $g\in C^\infty_0(\Omega)$,
then for a.e. $y\in \Omega\setminus \overline{\operatorname{supp}g}$, we have
$$
{p}(y)=-\int_\Omega \cG(x,y)\cdot f(x)\,dx+\int_{\Omega}\cP(x,y)g(x)\,dx.
$$
\end{enumerate}
The Green function for the adjoint operator $\cL^*$ is defined similarly.
\end{definition}

We note that in \eqref{180212@A1}, the divergence equation is understood as
$$
\int_\Omega \cG(\cdot,y) \cdot  \nabla \varphi\,dx=-\varphi(y)+(\varphi)_\Omega, \quad \forall \varphi\in C^\infty_0(\Omega).
$$
Since $\cG(\cdot,y)\in Y^{1,2}(\Omega\setminus \overline{B_r(y)})^d$ for any $r>0$,
the above identity implies that
$$
\operatorname{div} \cG(\cdot,y)=-\frac{1}{|\Omega|} \quad \text{for a.e. }\, x\in \Omega\setminus \{y\}.
$$
We also note that the property $(c)$ in Definition \ref{D0} together with the unique solvability of the problem \eqref{180206@eq2} gives the uniqueness of a Green function.
More precisely, under Assumption \ref{A1} below, by the solvability result in Lemma \ref{180108@lem1}, we get the uniqueness of a Green function in the sense that
if $(\tilde{\cG}, \tilde{\cP})$ is another Green function satisfying the properties in Definition \ref{D0}, then
for any $\phi\in C^\infty_0(\Omega)^d$ and $\varphi\in C^\infty_0(\Omega)$, we have
$$
\int_\Omega (\cG(x,y)-\tilde{\cG}(x,y))\cdot \phi(x)\,dx=\int_\Omega (\cP(x,y)-\tilde{\cP}(x,y))\varphi(x)\,dx=0
$$
for a.e. $y\in \Omega\setminus  \overline{\operatorname{supp}\varphi}$.

%========================================
\section{Main results}	\label{S3}
%========================================

In this section, we state our main results concerning Green function for the pressure of Stokes system.
For this, we impose the following solvability assumption of the divergence equation.

\begin{assumption}		\label{A1}
There exists a constant $K_0>0$ such that the following holds:
For any $g\in \tilde{L}^2(\Omega)$, there exists $u\in Y^{1,2}_0(\Omega)^d$ satisfying
$$
\operatorname{div} u=g \, \text{ in }\, \Omega, \quad \|Du\|_{L^2(\Omega)}\le K_0 \|g\|_{L^2(\Omega)}.
$$
\end{assumption}

\begin{remark}
It is known that Assumption \ref{A1} holds in a bounded John domain; see \cite[Theorem 4.1]{MR2263708}.
Here and throughout the paper, a domain is said to be bounded if it has finite diameter.
Note that a bounded domain $\Omega$ having a $C^{1,\rm{Dini}}$ boundary as in Definition \ref{D6} is a John domain as in \cite[Definition 2.1]{MR2263708} with respect to $(x_0, L)=(x_0, L)(d, R_0, \varrho_0)$.
Thus by \cite[Theorem 4.1]{MR2263708}, $\Omega$ satisfies Assumption \ref{A1} with $K_0=K_0(d,R_0,\varrho_0,\operatorname{diam}(\Omega))$.

A simple example of unbounded domain satisfying Assumption \ref{A1} is the whole space.
Indeed, based on a scaling argument and the existence of solutions to the divergence equation in a ball, one can verify that Assumption \ref{A1} holds with $K_0=K_0(d)$ when $\Omega=\bR^d$ and $d\ge 3$.
By the same reasoning, Assumption \ref{A1} holds when
$$
\Omega=\{x=(x_1,\ldots,x_d)\in \bR^d: x_1>0, \, x_2>0, \, \text{ or }\, x_d>0\}, \quad d\ge 2.
$$
Exterior domains with Lipschitz boundary also satisfy the assumption; see \cite[Theorem III.3.6, p. 189]{MR2808162}.
\end{remark}

In the theorem below and throughout the paper, we denote
$$
d_y=\operatorname{dist}(y,\partial \Omega), \quad d_y^*=\min\{1,d_y\}.
$$
Note that $d_y=\infty$ if $\Omega=\bR^d$.

\begin{theorem}		\label{M1}
Let $d\ge 2$ and $\Omega$ be a domain in $\bR^d$.
When $d=2$, $\Omega$ is assumed to be a Green domain satisfying
\begin{equation}		\label{180619@eq3}
|B_R(x)\setminus \Omega|\ge \theta R^2, \quad \forall x\in \partial \Omega, \quad \forall R\in (0,\infty).
\end{equation}
Suppose that the coefficients $A^{\alpha\beta}$ of $\cL$ are of partially Dini mean oscillation with respect to $x'$ in the interior of $\Omega$ satisfying Definition \ref{D1} $(a)$ with a Dini function $\omega=\omega_A$.
Then, under Assumption \ref{A1},  there exists a unique Green function $(\cG,\cP)$ for the pressure of $\cL$ in $\Omega$ such that for any $y\in \Omega$,
$$
\cG(\cdot,y) \,\text{ is continuous in }\, \Omega\setminus \{y\}
$$
and
$$
(\cG(\cdot,y), \cP(\cdot,y))\in W^{1,\infty}_{\rm{loc}}(\Omega\setminus \{y\})^d\times L^\infty_{\rm{loc}}(\Omega\setminus \{y\}).
$$
Moreover, for any $x,y\in \Omega$ with
$0<|x-y|\le d_y^*/2$,
we have
\begin{equation}	\label{180207@eq1a}
|\cG(x,y)|\le C|x-y|^{1-d},
\end{equation}
\begin{equation}	\label{180302@eq1}
\operatorname*{ess\, sup}_{B_{|x-y|/4}(x)} \big(|D \cG(\cdot,y)|+|\cP(\cdot ,y)|\big)\le C' |x-y|^{-d},
\end{equation}
where
$(C, C')=(C,C')(d,\lambda,\omega_A, K_0)$ and $C$ depends also on $\theta$ when $d=2$.
The same results hold if $\cL$ is replaced with its adjoint operator $\cL^*$.
\end{theorem}

\begin{remark}		\label{180217@rmk5}
From the proof of Theorem \ref{M1}, we get the following estimates for all $y\in \Omega$.
\begin{enumerate}[$(a)$]
\item
For any $R\in (0, d_y^*]$, we have that
$$
\|\cG(\cdot,y)\|_{L^{q^*}(\Omega\setminus \overline{B_R(y)})}\le C R^{-d/2},
$$
$$
\|D\cG(\cdot,y)\|_{L^2(\Omega\setminus \overline{B_R(y)})}+\|\cP(\cdot,y)\|_{L^2(\Omega\setminus \overline{B_R(y)})} \le C'R^{-d/2},
$$
where $q^*=2d/(d-2)$ if $d\ge 3$ and $q^*=\infty$ if $d=2$.
\item
We have that
$$
\big|\{x\in \Omega:|\cG(x,y)|>c_0t\}\big|\le C t^{-d/(d-1)}, \quad t\ge (d_y^*)^{1-d},
$$
$$
\big|\{x\in \Omega: |D_x\cG(x,y)|>t\}\big|+\big|\{x\in \Omega:|\cP(x,y)|>t\}\big|\le C' t^{-1}, \quad t\ge (d_y^*)^{-d},
$$
where $c_0=1$ if $d\ge 3$ and $c_0=c_0(\lambda, \omega_A, K_0,\theta)>0$ if $d=2$.
\item
For any $R\in (0, d_y^*]$, we have that
$$
\|\cG(\cdot,y)\|_{L^q(B_R(y))}\le C_q R^{1-d+d/q}, \quad \forall\ q\in (0,d/(d-1)),
$$
$$
\|D\cG(\cdot,y)\|_{L^q(B_R(y))}+\|\cP(\cdot,y)\|_{L^q(B_R(y))}\le C_q' R^{-d+d/q}, \quad \forall\ q\in (0,1).
$$
\end{enumerate}
In the above, $(C,C')=(C,C')(d,\lambda,\omega_A, K_0)$, $(C_q,C_q')=(C_q,C_q')(d,\lambda,\omega_A, K_0, q)$, and $(C,C_q)$ depends also on $\theta$ when $d=2$.
\end{remark}

\begin{remark}		\label{180718@rmk2}
In Theorem \ref{M1},
if the coefficients $A^{\alpha\beta}$ of $\cL$ are of Dini mean oscillation with respect to {\em{all}} the directions in $\Omega$ satisfying Definition \ref{D1} $(b)$,
then by Definition \ref{D0} $(b)$ and \eqref{180406@eq1}, we see that $D\cG(\cdot,y)$ and $\cP(\cdot,y)$ are continuous in ${\Omega}\setminus\{y\}$.
Hence, $``\operatorname{ess\, sup}"$ in \eqref{180302@eq1} can be replaced with $``\sup"$.
Therefore, for any $x,y\in \Omega$ with $0<|x-y|\le d_y^*/2$, we have
$$
|D_x\cG(x,y)|+|\cP(x,y)|\le C |x-y|^{-d},
$$
where $C=C(d,\lambda, \omega_A, K_0)$.
\end{remark}

\begin{remark}		\label{190118@rmk1}
In the case when $|\Omega|<\infty$, the condition \eqref{180619@eq3} can be replaced with the condition
\begin{equation}		\label{181104@eq1}
|B_R(x)\setminus \Omega|\ge \theta_0R^2, \quad \forall x\in \partial \Omega, \quad \forall R\in (0,1].
\end{equation}
Indeed, \eqref{180619@eq3} and \eqref{181104@eq1} are equivalent because if \eqref{181104@eq1} holds, then \eqref{180619@eq3} also holds with $\theta=\theta(\theta_0, |\Omega|)$.
\end{remark}

In the next theorem, we prove the global pointwise estimates for the Green function and its derivatives in a
domain having a $C^{1,\rm{Dini}}$ boundary.

\begin{theorem}		\label{M2}
Let $d\ge 2$ and  $\Omega$ be a domain in $\bR^d$ having a $C^{1,\rm{Dini}}$ boundary as in Definition \ref{D6}.
When $d=2$, $\Omega$ is assumed to be a Green domain satisfying \eqref{180619@eq3}.
Suppose that the coefficients $A^{\alpha\beta}$ of $\cL$ are of Dini mean oscillation in $\Omega$ satisfying Definition \ref{D1} $(b)$ with a Dini function $\omega=\omega_A$.
Let $(\cG,\cP)$ be the Green function for the pressure of $\cL$ constructed in Theorem \ref{M1} under Assumption \ref{A1}.
Then for any $y\in \Omega$,
$$
\cG(\cdot,y)\, \text{ is continuously differentiable in }\, \overline{\Omega}\setminus \{y\}
$$
and
$$
\cP(\cdot,y)\, \text{ is continuous in }\, \overline{\Omega}\setminus \{y\}.
$$
Moreover, for any $x,y\in \Omega$ with $0<|x-y|\le 1$, we have
\begin{equation}		\label{180208@eq1}
|\cG(x,y)|\le C|x-y|^{1-d},
\end{equation}
\begin{equation}		\label{180217@C1}
|D_x\cG(x,y)|+|\cP(x,y)|\le C |x-y|^{-d},
\end{equation}
where
$C=C(d,\lambda, \omega_A, K_0,  R_0, \varrho_0)$ and $C$ depends also on $\theta$ when $d=2$.
Furthermore,
if $(\cG^*, \cP^*)$ is the Green function for the pressure of $\cL^*$ in $\Omega$, then for any $y\in \Omega$, there exists a measure zero set $N_y\subset \Omega$ containing $y$ such that
\begin{equation}		\label{180220@eq3}
\cP(x,y)=\cP^*(y,x) \quad \text{for all }\, x\in \Omega\setminus N_y.
\end{equation}
\end{theorem}

\begin{remark}		\label{180717@rmk1}
Note that any domain $\Omega\subset \bR^d$ having a $C^{1,\rm{Dini}}$ boundary as in Definition \ref{D6} satisfies
\begin{equation}		\label{180703@eq5}
|B_R(x)\setminus \Omega|\ge \theta R^d, \quad \forall x\in \partial \Omega, \quad \forall R\in (0,1],
\end{equation}
where $\theta=\theta(d,R_0, \varrho_0)$.
Therefore, by Remark \ref{190118@rmk1}, the condition \eqref{180619@eq3} can be removed in Theorem \ref{M2} when $d=2$ and $|\Omega|<\infty$.
In this case, the constant $C$ in \eqref{180208@eq1} and \eqref{180217@C1} also depends on $|\Omega|$ instead of $\theta$.

\end{remark}

\begin{remark}		\label{180326@rmk1}
Because the Green function satisfies the zero Dirichlet boundary condition, we have a better estimate for $\cG$ than \eqref{180208@eq1} near the boundary of $\Omega$.
Indeed, for any $x,y\in \Omega$ with $ 0<|x-y|\le 1$, we have
$$
|\cG(x,y)|\le C\min\{d_x,|x-y|\}\cdot  |x-y|^{-d},
$$
where $C=C(d,\lambda,\omega_A, K_0, R_0, \varrho_0)$ and $C$ depends also on $\theta$ when $d=2$.
For further details, see the proof of Theorem \ref{M2} in Section \ref{S_Pr}.
\end{remark}

\begin{remark}		\label{180218@rmk1}
From the proof of Theorem \ref{M2}, we get the following estimates for any $y\in \Omega$.
\begin{enumerate}[$(a)$]
\item
For any $R\in (0, 1]$, we have that
$$
\|\cG(\cdot,y)\|_{L^{q^*}(\Omega\setminus \overline{B_R(y)})}\le C R^{-d/2},
$$
$$
\|D\cG(\cdot,y)\|_{L^2(\Omega\setminus \overline{B_R(y)})}+\|\cP(\cdot,y)\|_{L^2(\Omega\setminus \overline{B_R(y)})} \le C'R^{-d/2},
$$
where $q^*=2d/(d-2)$ if $d\ge 3$ and $q^*=\infty$ if $d=2$.
\item
For any $t>0$, we have that
$$
\big|\{x\in \Omega:|\cG(x,y)|>t\}\big|\le C t^{-d/(d-1)},
$$
$$
\big|\{x\in \Omega: |D_x\cG(x,y)|>t\}\big|+\big|\{x\in \Omega:|\cP(x,y)|>t\}\big|\le C' t^{-1}.
$$
\item
For any $R\in (0,1]$, we have that
$$
\|\cG(\cdot,y)\|_{L^q(B_R(y))}\le C_q R^{1-d+d/q}, \quad \forall\ q\in (0, d/(d-1)),
$$
$$
\|D\cG(\cdot,y)\|_{L^q(B_R(y))}+\|\cP(\cdot,y)\|_{L^q(B_R(y))}\le C_q' R^{-d+d/q}, \quad \forall\ q\in (0,1).
$$
\end{enumerate}
In the above,
$$
(C,C')=(C,C')(d,\lambda, \omega_A, K_0, R_0, \varrho_0), \,\, (C_q,C_q')=(C_q,C_q')(d,\lambda, \omega_A, K_0, R_0, \varrho_0,q),
$$
and $(C,C_q)$ depends also on $\theta$ when $d=2$.
\end{remark}

We finish this section with the following theorem, where we extend the results in Theorem \ref{M1} up to the boundary in a half space $\bR^d_+$, defined by
$$
\bR^d_+=\{x=(x_1,x')\in \bR^d: x_1>0, \, x'\in \bR^{d-1}\}, \quad d\ge 2,
$$
when  the measurable direction of the coefficients is perpendicular to the boundary $\partial \bR^d_+$.
For the reader's convenience and future reference, we present the results together with the case when $\Omega=\bR^d$ with $d\ge 3$.

\begin{definition}		\label{D3}
Let $f\in L^1_{\rm{loc}}(\overline{\Omega})$, where $\Omega=\bR^d$ or $\bR^d_+$.
We say that $f$ is of {\em{partially Dini mean oscillation with respect to $x'$}} in $\Omega$ if there exists a Dini function $\omega:(0,1]\to [0,\infty)$ such that for any $x\in \overline{\Omega}$ and $r\in (0,1]$, we have
$$
\dashint_{\Omega_r(x)}\bigg|f(y)-\dashint_{B_r'(x')}f(y_1,z')\,dz'\bigg|\,dy\le \omega(r).
$$
\end{definition}

\begin{theorem}		\label{M5}
Let $\Omega=\bR^d$ with $d\ge 3$ or $\Omega=\bR^d_+$ with $d\ge 2$.
Suppose that the coefficients $A^{\alpha\beta}$ of $\cL$ are of partially Dini mean oscillation with respect to $x'$ in $\Omega$ satisfying Definition \ref{D3} with a Dini function $\omega=\omega_A$.
Then there exists a unique Green function $(\cG, \cP)$ for the pressure of $\cL$ in $\Omega$ such that for any $y\in \Omega$ and $r>0$,
$$
\cG(\cdot,y) \, \text{ is continuous in }\, \overline{\Omega}\setminus \{y\}
$$
and
\begin{equation}		\label{180724@A1}
(\cG(\cdot,y), \cP(\cdot,y))\in W^{1,\infty}(\Omega\setminus \overline{B_r(y)})^d\times L^\infty(\Omega\setminus \overline{B_r(y)}).
\end{equation}
Moreover, for any $x,y\in \Omega$ satisfying $0<|x-y|\le 1$, we have
\begin{equation}		\label{180227@eq1}
|\cG(x,y)|\le C\min\{d_x,|x-y|\}\cdot |x-y|^{-d},
\end{equation}
\begin{equation}		\label{180302@eq2}
\operatorname*{ess\,sup}_{\Omega_{|x-y|/4}(x)} \big(|D \cG(\cdot,y)|+|\cP(\cdot,y)|\big)\le C|x-y|^{-d},
\end{equation}
where $C=C(d,\lambda, \omega_A)$.
Furthermore, the same estimates in Remark \ref{180218@rmk1} hold with $C=C(d,\lambda, \omega_A)$ and $C_q=C_q(d,\lambda, \omega_A, q)$.
The same results hold if $\cL$ is replaced with its adjoint operator $\cL^*$.
\end{theorem}

%========================================
\section{Preliminary results}		\label{S_Pre}
%========================================

In this section, we prove some preliminary results.
We do not impose any regularity assumptions on the coefficients $A^{\alpha\beta}$ of the operator $\cL$.
The following lemma concerns the solvability  of Stokes system in $Y^{1,2}_0(\Omega)\times \tilde{L}^2(\Omega)$.

\begin{lemma}		\label{180108@lem1}
Let $d\ge 2$ and  $\Omega$ be a domain in $\bR^d$.
When $d=2$, $\Omega$ is assumed to be a Green domain.
Suppose that $f_\alpha\in L^2(\Omega)^d$, $g\in \tilde{L}^2(\Omega)$, and
$$
f\in
\left\{
\begin{aligned}
L^{2d/(d+2)}(\Omega)^d &\quad \text{if }\, d\ge 3,\\
L^q(\Omega)^2 \, \text{for some $q>1$ with a compact support} &\quad \text{if }\, d=2.
\end{aligned}
\right.
$$
Then, under Assumption \ref{A1}, there exists a unique $(u,p)\in Y^{1,2}_0(\Omega)^d\times \tilde{L}^2(\Omega)$ satisfying
$$
\left\{
\begin{aligned}
\cL u+\nabla p=f+D_\alpha f_\alpha &\quad \text{in }\, \Omega,\\
\operatorname{div} u=g &\quad \text{in }\, \Omega,\\
u=0 &\quad \text{on }\, \partial \Omega.
\end{aligned}
\right.
$$
Moreover, we have that for $d\ge 3$,
\begin{equation}		\label{190101@eq1}
\|Du\|_{L^2(\Omega)}+\|p\|_{L^2(\Omega)}\le C\big(\|f\|_{L^q(\Omega)}+\|f_\alpha\|_{L^2(\Omega)}+\|g\|_{L^2(\Omega)}\big),
\end{equation}
and that for $d=2$,
$$
\|Du\|_{L^2(\Omega)}+\|p\|_{L^2(\Omega)}\le C_q\|f\|_{L^q(\Omega)}+C\big(\|f_\alpha\|_{L^2(\Omega)}+\|g\|_{L^2(\Omega)}\big),
$$
where $C=C(d,\lambda,K_0)$ and $C_q=C_q(\lambda,K_0, \Omega, \operatorname{supp}f,q)$.
\end{lemma}

\begin{proof}
We only present here the detailed proof of the case when $d=2$ and $|\Omega|=\infty$ because the others are the same as the proof of \cite[Lemma 3.2]{MR3693868}, where the authors proved the $W^{1,2}$-solvability in a bounded domain.

Let $f$, $f_\alpha$, and $g$ satisfy the hypothesis of the lemma.
Suppose that $\operatorname{supp}f\subset B_{r_0}=B_{r_0}(0)$ for some $r_0>0$.
Since $f$ has a compact support, we may assume that $q\in (1,2)$.
We define a Hilbert space $H(\Omega)$ as the set of all functions $u\in Y^{1,2}_0(\Omega)^2$ satisfying $\operatorname{div} u=0$ in $\Omega$.
Let $H^\bot(\Omega)$ be the orthogonal complement of $H(\Omega)$ in $Y^{1,2}_0(\Omega)^2$ and $P$ be the orthogonal projection from $Y^{1,2}_0(\Omega)^2$ onto
$H^\bot(\Omega)$.
By Assumption \ref{A1}, there exists $w\in Y^{1,2}_0(\Omega)^2$ such that
$$
\operatorname{div}(Pw)=\operatorname{div} w=g \, \text{ in }\, \Omega,
$$
and
\begin{equation}		\label{180726@eq5}
\|D(Pw)\|_{L^2(\Omega)}\le \|Dw\|_{L^2(\Omega)}\le K_1 \|g\|_{L^2(\Omega)}.
\end{equation}
Now we set
$$
\cF(\phi)=-\int_\Omega f\cdot \phi\,dx+\int_\Omega f_\alpha \cdot D_\alpha \phi\,dx-\int_\Omega A^{\alpha\beta} D_\beta (Pw)\cdot D_\alpha \phi\,dx, \quad \phi\in H(\Omega).
$$
By H\"older's inequality, the Poincar\'e inequality, and \cite[Lemma 1.84]{MR1461542}, we have
\begin{align*}
\bigg|\int_\Omega f\cdot \phi\,dx\bigg|
&\le \|f\|_{L^q(\Omega_{r_0})}\big\|\phi\chi_{B_{r_0}}-\big(\phi \chi_{B_{r_0}}\big)_{B_{r_0}}\big\|_{L^{q/(q-1)}(B_{r_0})}\\
&\quad +C\|f\|_{L^q(\Omega_{r_0})}\|\phi\|_{L^1(\Omega_{r_0})}\\
&\le C\|f\|_{L^q(\Omega_{r_0})} \big\|D\big(\phi \chi_{B_{r_0}}\big)\big\|_{L^{2q/(3q-2)}(B_{r_0})}+C\|f\|_{L^q(\Omega_{r_0})} \|\phi\|_{L^2(\Omega_{r_0})}\\
&\le C_0\|f\|_{L^q(\Omega_{r_0})} \|D\phi\|_{L^2(\Omega)},
\end{align*}
where $C_0=C_0(\Omega, B_{r_0},q)$.
From this together with \eqref{180726@eq5}, we see that $\cF$ is a bounded linear functional on $H(\Omega)$ satisfying
$$
|\cF(\phi)|\le CM \|D\phi\|_{L^2(\Omega)}, \quad M:= C_0 \|f\|_{L^q(\Omega)}+\|f_\alpha\|_{L^2(\Omega)}+\|g\|_{L^2(\Omega)},
$$
where $C=C(\lambda, K_1)$.
Thus by the Lax-Milgram theorem, there exists  $v\in H(\Omega)$ satisfying
$$
\|Dv\|_{L^2(\Omega)}\le C M, \quad \int_\Omega A^{\alpha\beta}D_\beta v\cdot D_\alpha \phi\,dx=\cF(\phi), \quad \forall \phi\in H(\Omega).
$$
Therefore, the function $u=v+Pw$ satisfies
\begin{equation}		\label{180726@eq7}
\operatorname{div} u=g \, \text{ in }\, \Omega, \quad \|Du\|_{L^2(\Omega)}\le CM,
\end{equation}
\begin{equation}		\label{180726@eq7a}
\int_{\Omega}A^{\alpha\beta}D_\beta u \cdot D_\alpha \phi\,dx=-\int_\Omega f\cdot \phi\,dx+\int_\Omega f_\alpha \cdot D_\alpha \phi\,dx, \quad \forall \phi \in H(\Omega).
\end{equation}
To find $p$, we set
$$
\cK( \varphi)=-\int_{\Omega}A^{\alpha\beta} D_\beta u\cdot D_\alpha (P \Phi)\,dx-\int_\Omega f\cdot (P\Phi)\,dx+\int_\Omega f_\alpha \cdot D_\alpha (P\Phi)\,dx,
$$
where $\varphi \in L^2(\Omega)$ and $\Phi\in Y^{1,2}_0(\Omega)^2$ such that
$$
\operatorname{div}(P\Phi)=\operatorname{div}\Phi= \varphi \, \text{ in }\, \Omega
$$
and
$$
\|D(P\Phi)\|_{L^2(\Omega)}\le \|D\Phi\|_{L^2(\Omega)}\le K_1 \|\varphi\|_{L^2(\Omega)}.
$$
Then it can be easily seen that $\cK$ is a bounded linear functional in $L^2(\Omega)$ with the estimate
$$
|\cK( \varphi)|\le C(\|Du\|_{L^2(\Omega)}+M)\|\varphi\|_{L^2(\Omega)}
\le CM\|\varphi\|_{L^2(\Omega)}.
$$
Therefore, by the Riesz representation theorem, there exists $p\in L^2(\Omega)$ such that
$$
\|p\|_{L^2(\Omega)}\le CM, \quad \int_\Omega p \varphi\,dx=\cK(\varphi), \quad \forall \varphi\in L^2(\Omega).
$$
By the definition of $\cK$ and $P$, we have
$$
\int_\Omega A^{\alpha\beta} D_\beta u \cdot D_\alpha \phi\,dx+\int_\Omega p\operatorname{div} \phi\,dx=-\int_\Omega f\cdot \phi\,dx+\int_\Omega f_\alpha \cdot D_\alpha \phi\,dx, \quad \forall \phi\in H^{\bot}(\Omega).
$$
This together with \eqref{180726@eq7} and \eqref{180726@eq7a} proves the lemma when $d=2$.
\end{proof}

In the two dimensional case, the $L^2$-estimate in Lemma \ref{180108@lem1} is not well suited to proving optimal estimates of Green functions.
Hence, instead of the $L^2$-estimate, we shall use a $L^q$-estimate for some $q>2$ (see Lemma \ref{180913@lem1} below), which is an easy consequence of the following reverse H\"older's inequality.

\begin{lemma}		\label{180723@lem1}
Let $\Omega$ be a Green domain in $\bR^2$ satisfying \eqref{180619@eq3}.
Then, under Assumption \ref{A1}, there exists $q_0=q_0(\lambda,K_0, \theta)>2$ such that if
$(u, p)\in Y^{1,2}_0(\Omega)^d\times \tilde{L}^2(\Omega)$ satisfies
$$
\left\{
\begin{aligned}
\cL u+\nabla p=D_\alpha f_\alpha &\quad \text{in }\, \Omega,\\
\operatorname{div} u=g &\quad \text{in }\, \Omega,\\
u=0 &\quad \text{on }\, \partial \Omega,
\end{aligned}
\right.
$$
where $f_\alpha\in L^{q_0}(\Omega)^2$ and $g\in \tilde{L}^{q_0}(\Omega)$, then for any $x\in \overline{\Omega}$ and $R\in (0,\infty)$, we have
\begin{equation}		\label{181026@eq1}
\big(|Du|^{q_0}+|p|^{q_0}\big)^{1/q_0}_{\Omega_{R/2}(x)}
\le C\big(|Du|^2+|p|^2\big)^{1/2}_{\Omega_{R}(x)}
+C\big(|f_\alpha |^{q_0}+|g|^{q_0}\big)^{1/q_0}_{\Omega_{R}(x)},
\end{equation}
where $C=C(\lambda, K_0, \theta)$.
\end{lemma}

\begin{proof}
For the proof of the lemma, we refer the reader to that of \cite[Lemma 3.8]{ MR3758532}, where the authors proved the reverse H\"older's inequality in a bounded Reifenberg flat domain.
The argument in the proof of \cite[Lemma 3.8]{ MR3758532} is general enough to allow domains  to satisfy   \eqref{180619@eq3} and Assumption \ref{A1}.

We note that our statement is slightly different from that of \cite[Lemma 3.8]{ MR3758532}.
In \cite[Lemma 3.8]{ MR3758532}, the exponent $q_0\in (2,q_1)$ depends also on $q_1$ under the assumption that the data are $q_1$-th integrable.
Indeed, if $q_1$ is sufficiently close to $2$, then $q_0$ can be chosen as $q_1$.
This follows by using Proposition \ref{181101@prop1} below instead of \cite[Proposition 3.7]{MR3758532}.
\end{proof}

\begin{proposition}		\label{181101@prop1}
Let $1<q<s$, $\Phi\in L^{q}(Q)$, and  $\Psi\in L^{s}(Q)$ be nonnegative functions, where $Q$ is a $d$-dimensional cube.
Suppose that
$$
\dashint_{B_R(x_0)} \Phi^{q}\,dx\le C_0 \bigg(\dashint_{B_{8R}(x_0) } \Phi\,dx\bigg)^{q}+\dashint_{B_{8R}(x_0)} \Psi^q\,dx+\delta \dashint_{B_{8R}(x_0)}\Phi^{q}\,dx
$$
for any $B_{8R}(x_0)\subset Q$ with $R\in (0, R_1]$,
where $R_1>0$, $C_0>1$, and $\delta \in [0,1)$.
Then there exist positive constants $\varepsilon$ and $C$, depending only on $d$, $q$, $C_0$, and $\delta$, such that
$$
\Phi\in L^{q_0}_{\rm{loc}}(Q)  \quad \forall q_0\in [q,\min\{s, q+\varepsilon\}]
$$
and
$$
\bigg(\dashint_{B_R(x_0)}\Phi^{q_0}\,dx\bigg)^{1/q_0}\le C\bigg(\dashint_{B_{8R}(x_0)}\Phi^{q}\,dx\bigg)^{1/q}+C\bigg(\dashint_{B_{8R}(x_0)} \Psi^{q_0}\,dx\bigg)^{1/q_0}
$$
for any $B_{8R}(x_0)\subset Q$ with $R\in (0, R_1]$.
\end{proposition}

\begin{proof}
See, for instance, \cite[Ch.V]{MR0717034} for the proof of the proposition.
\end{proof}

\begin{lemma}		\label{180913@lem1}
Let $\Omega$ be a Green domain in $\bR^2$ satisfying \eqref{180619@eq3}.
Let
$$
q_0=q_0(\lambda,K_0,\theta)>2
$$
be the constant from Lemma \ref{180723@lem1} under Assumption \ref{A1}.
If $(u, p)\in Y^{1,2}_0(\Omega)^d\times \tilde{L}^2(\Omega)$ satisfies
\begin{equation}		\label{181104@A1}
\left\{
\begin{aligned}
\cL u+\nabla p=f+D_\alpha f_\alpha &\quad \text{in }\, \Omega,\\
\operatorname{div} u=g &\quad \text{in }\, \Omega,\\
u=0 &\quad \text{on }\, \partial \Omega,
\end{aligned}
\right.
\end{equation}
where $f\in L^{2q_0/(2+q_0)}(\Omega)^2$, $f_\alpha\in L^{q_0}(\Omega)^2$, and $g\in \tilde{L}^{q_0}(\Omega)$, then we have
$$
\|Du\|_{L^{q_0}(\Omega)}+\|p\|_{L^{q_0}(\Omega)}\le C\big(\|f\|_{L^{2q_0/(2+q_0)}(\Omega)}+\|f_\alpha\|_{L^{q_0}(\Omega)}+\|g\|_{L^{q_0}(\Omega)}\big),
$$
where $C=C(\lambda, K_0,\theta)$.
\end{lemma}

\begin{proof}
If $f\equiv 0$, then the lemma follows by letting $R\to \infty$ in \eqref{181026@eq1}.
On the other hand, if $f\not\equiv 0$, then by the existence of solutions to the divergence equation in the whole space (see, for instance, \cite[Lemma 3.2]{MR3670039}), there exist functions $F_\alpha\in Y^{2q_0/(2+q_0)}(\Omega)^2$, $\alpha\in \{1,2\}$, satisfying
$$
D_\alpha F_\alpha=f \chi_{\Omega} \quad \text{in }\, \bR^2
$$
and
\begin{equation}		\label{181104@B1}
\|F_\alpha\|_{L^{q_0}(\bR^2)}+\|DF_\alpha\|_{L^{2q_0/(2+q_0)}(\bR^2)}\le C \|f\|_{L^{2q_0/(2+q_0)}(\Omega)},
\end{equation}
where $C$ is an universal constant.
Since $(u, p)$ satisfies \eqref{181104@A1} with $D_\alpha (F_\alpha+f_\alpha)$ in place of $f+D_\alpha f_\alpha$, we get the desired estimate from \eqref{181104@B1}.
The lemma is proved.
\end{proof}

%========================================
\section{Approximated Green function}		\label{S_Ap}	
%========================================
Hereafter in the paper, we shall use the following notation.

\begin{notation}
For a given function $f$, if there is a continuous version of $f$, that is, there is a continuous function $\tilde{f}$ such that $\tilde{f}=f$ in the almost everywhere sense, then we replace $f$ with $\tilde{f}$ and denote the version again by $f$.
\end{notation}

In this section, we assume that the hypotheses in Theorem \ref{M1} hold.
Under these hypotheses, we shall construct approximated Green functions for the pressure of the Stokes system.
We recall the notation that if $|\Omega|=\infty$, then $1/|\Omega|=0$ and
$(u)_\Omega=0$ for any $u\in L^q(\Omega)$.

Let us fix a smooth function $\Phi$ defined on $\bR^d$ such that
$$
0\le \Phi\le 2, \quad \operatorname{supp}\Phi\subset B_1(0),\quad  \int_{\bR^d}\Phi\,dx=1.
$$
For $y\in \Omega$ and $\varepsilon\in (0,1]$, we define
$$
\Phi_{\varepsilon,y}(x)=\varepsilon^{-d}\Phi((x-y)/\varepsilon).
$$
By Lemma \ref{180108@lem1}, there exists a unique $(\cG_\varepsilon(\cdot,y),\cP_\varepsilon(\cdot,y))\in Y^{1,2}_0(\Omega)^d\times \tilde{L}^2(\Omega)$ satisfying
\begin{equation}		\label{180110@eq2}
\left\{
\begin{aligned}
\cL \cG_\varepsilon(\cdot,y) +\nabla \cP_\varepsilon(\cdot,y)=0 &\quad \text{in }\, \Omega,\\
\operatorname{div} \cG_\varepsilon(\cdot,y)=\Phi_{\varepsilon,y}-(\Phi_{\varepsilon,y})_\Omega &\quad \text{in }\, \Omega,\\
\cG_{\varepsilon}(\cdot,y)=0 &\quad \text{on }\, \partial \Omega.
\end{aligned}
\right.
\end{equation}
Moreover, using the fact that
$$
\|\Phi_{\varepsilon,y}-(\Phi_{\varepsilon,y})_{\Omega}\|_{L^2(\Omega)}\le 2 \|\Phi_{\varepsilon,y}\|_{L^2(\Omega)}\le C\varepsilon^{-d/2},
$$
we have
\begin{equation}		\label{180110@eq4}
\|D\cG_\varepsilon(\cdot,y)\|_{L^2(\Omega)}+\|\cP_\varepsilon(\cdot,y)\|_{L^2(\Omega)}
\le C\varepsilon^{-d/2},
\end{equation}
where $C=C(d,\lambda, K_0)$.
Throughout the paper, we call $(\cG_\varepsilon(\cdot,y), \cP_\varepsilon(\cdot,y))$ {\emph{the approximated Green function for the pressure of $\cL$ in $\Omega$}}.

In the lemma below, we prove a $L^1$-estimate for $\cG_{\varepsilon}(\cdot,y)$.

\begin{lemma}		\label{180110@lem1}
Let $y\in \Omega$ and $\varepsilon\in (0,1]$.
Then for any $x\in \Omega$ and $R\in (0, d_{x}^*]$, we have
$$
\|\cG_\varepsilon(\cdot,y)\|_{L^1(B_R(x))} \le C R,
$$
where $C=C(d,\lambda, \omega_A, K_0)$ and $C$ depends also on $\theta$ when $d=2$.
\end{lemma}

\begin{proof}
We consider the following two cases:
$$
2\varepsilon\le R \quad \text{and}\quad 2\varepsilon>R.
$$
\begin{enumerate}[i.]
\item
$2\varepsilon\le R$:
Set
$$
f=\chi_{B_R(x)}\big(\operatorname{sgn}\cG_\varepsilon^1(\cdot,y), \ldots, \operatorname{sgn}\cG_\varepsilon^d(\cdot,y)\big)^{\top}.
$$
Then by Lemma \ref{180108@lem1}, there exists a unique $(u,p)\in Y^{1,2}_0(\Omega)^d\times \tilde{L}^2(\Omega)$ satisfying
\begin{equation}		\label{180110@eq4a}
\left\{
\begin{aligned}
\cL^* u+\nabla p=-f &\quad \text{in }\, \Omega,\\
\operatorname{div} u=0 &\quad \text{in }\, \Omega,\\
u=0 &\quad \text{on }\, \partial \Omega.
\end{aligned}
\right.
\end{equation}
We apply $u$ and $\cG_\varepsilon(\cdot,y)$ as test functions to  \eqref{180110@eq2} and \eqref{180110@eq4a}, respectively, to get
$$
\int_\Omega f \cdot \cG_\varepsilon(\cdot,y)\,dx=\int_\Omega p\,  \Phi_{\varepsilon,y}\,dx,
$$
which implies that (using $\varepsilon\le R/2$)
$$
\|\cG_\varepsilon(\cdot,y)\|_{L^1(B_{R}(x))}\le \|p\|_{L^\infty(B_{R/2}(y))}.
$$
Hence, we get from \eqref{180110@eq8} that
\begin{equation}		\label{180726@eq1}
\|\cG_{\varepsilon}(\cdot,y)\|_{L^1(B_{R}(x))} \le C R^{-d} \big(\|Du\|_{L^1(B_R(y))}+\|p\|_{L^1(B_R(y))}\big)+CR,
\end{equation}
where $C=C(d,\lambda, \omega_A)$.
If $d\ge 3$, then by \eqref{190101@eq1}, we have
$$
\|Du\|_{L^2(\Omega)}+\|p\|_{L^2(\Omega)}\le C R^{(d+2)/2} ,
$$
where $C=C(d,\lambda, K_0)$.
From this together with \eqref{180726@eq1} and H\"older's inequality, we get the desired estimate.
On the other hand, if $d=2$, then by Lemma \ref{180913@lem1} applied to  \eqref{180110@eq4a}, we have
$$
\|Du\|_{L^{q_0}(\Omega)}+\|p\|_{L^{q_0}(\Omega)}\le C\|f\|_{L^{2q_0/(2+q_0)}(\Omega)}\le CR^{(2+q_0)/q_0},
$$
where $C=C(\lambda, K_0,\theta)$ and $q_0=q_0( \lambda, K_0,\theta)>2$.
Thus, from \eqref{180726@eq1} and H\"older's inequality, we get the desired estimate.
\item
$2\varepsilon>R$:
If $d\ge 3$, then by H\"older's inequality, the Sobolev inequality, and \eqref{180110@eq4}, we have
$$
\begin{aligned}
&\|\cG_{\varepsilon}(\cdot,y)\|_{L^1(B_R(x))}
\le CR^{(d+2)/2}\|\cG_{\varepsilon}(\cdot,y)\|_{L^{2d/(d-2)}(B_R(x))}\\
&\le CR^{(d+2)/2}\|D\cG_{\varepsilon}(\cdot,y)\|_{L^2(\Omega)}\le CR^{(d+2)/2}\varepsilon^{-d/2}\le CR,
\end{aligned}
$$
which gives the desired estimate.

Now we assume $d=2$.
Let $f_\alpha\in C^\infty_0(\Omega)^2$ and $g\in C^\infty_0(\Omega)$.
Then by Lemma \ref{180108@lem1},
there exists a unique $(v, \pi)\in Y^{1,2}_0(\Omega)^2\times \tilde{L}^2(\Omega)$ satisfying
\begin{equation}		\label{180906@eq3}
\left\{
\begin{aligned}
\cL^* v+\nabla \pi=D_\alpha f_\alpha &\quad \text{in }\, \Omega,\\
\operatorname{div} v=g-(g)_{\Omega} &\quad \text{in }\, \Omega,\\
u=0 &\quad \text{on }\, \partial \Omega.
\end{aligned}
\right.
\end{equation}
We also get from Lemma \ref{180913@lem1} that
\begin{equation}		\label{180906@eq3a}
\|Dv\|_{L^{q_0}(\Omega)}+\|\pi\|_{L^{q_0}(\Omega)}
\le C\big(\|f_\alpha\|_{L^{q_0}(\Omega)}+\|g\|_{L^{q_0}(\Omega)}\big),
\end{equation}
where $C=C( \lambda, K_0, \theta)$ and $q_0=q_0(\lambda,K_0, \theta)>2$.
We test \eqref{180110@eq2} and \eqref{180906@eq3} with $v$ and $\cG_{\varepsilon}(\cdot,y)$, respectively, to obtain
$$
\int_{\Omega} \big(f_\alpha \cdot D_\alpha \cG_{\varepsilon}(\cdot,y)+g \cP_{\varepsilon}(\cdot,y)\big)\,dz=\int_{B_{\varepsilon}(y)} \pi \Phi_{\varepsilon, y}\,dz.
$$
Then by H\"older's inequality, \eqref{180906@eq3a}, and $R<2\varepsilon\le 2$, we have
$$
\begin{aligned}
\bigg|\int_{\Omega} \big(f_\alpha \cdot D_\alpha \cG_{\varepsilon}(\cdot,y)+g \cP_{\varepsilon}(\cdot,y)\big)\,dz\bigg|
& \le \|\Phi_{\varepsilon,y}\|_{L^{q_0/(q_0-1)}(B_\varepsilon(y))}\|\pi\|_{L^{q_0}(\Omega)} \\
&\le CR^{-2/q_0}\big(\|f_\alpha\|_{L^{q_0}(\Omega)}+\|g\|_{L^{q_0}(\Omega)}\big).
\end{aligned}
$$
Since the above inequality holds for all $f_\alpha\in C^\infty_0(\Omega)^2$ and $g\in C^\infty_0(\Omega)$, we get
$$
\|D\cG_{\varepsilon}(\cdot,y)\|_{L^{q_0/(q_0-1)}(\Omega)}+\|\cP_{\varepsilon}(\cdot,y)\|_{L^{q_0/(q_0-1)}(\Omega)} \le CR^{-2/q_0},
$$
and thus, we obtain by the Sobolev inequality that
$$
\|\cG_{\varepsilon}(\cdot,y)\|_{L^{2q_0/(q_0-2)}(\Omega)}\le CR^{-2/q_0}.
$$
This together with H\"older's inequality yields the desired estimate.
\end{enumerate}
The lemma is proved.
\end{proof}

We establish the following estimates  uniformly in $\varepsilon\in (0, 1]$.

\begin{lemma}		\label{180719@lem1}
Let $y\in \Omega$, $\varepsilon\in (0,1]$,  $R\in (0, d_y^*]$, and
$$
q^*=\frac{2d}{d-2} \quad \text{if }\, d\ge 3, \quad q^*=\infty \quad \text{if }\, d=2.
$$
Then we have
\begin{equation}		\label{180618@A1a}
\|\cG_\varepsilon(\cdot,y)\|_{L^{q^*}(\Omega\setminus \overline{B_R(y)})}\le C R^{-d/2}
\end{equation}
and
\begin{equation}		\label{180618@A1}
\|D\cG_\varepsilon(\cdot,y)\|_{L^2(\Omega\setminus \overline{B_R(y)})}+\|\cP_\varepsilon(\cdot,y)\|_{L^2(\Omega\setminus \overline{B_R(y)})} \le C'R^{-d/2},
\end{equation}
where
$(C,C')=(C,C')(d,\lambda, \omega_A, K_0)$ and $C$ depends also on $\theta$ when $d=2$.
\end{lemma}

\begin{proof}
We first prove the estimate \eqref{180618@A1}.
If $\varepsilon > R/2$, then \eqref{180618@A1} follows immediately from \eqref{180110@eq4}.
Now we assume $\varepsilon \le R/2$.
Set
$$
f_\alpha=\chi_{\Omega\setminus \overline{B_R(y)}} D_\alpha \cG_\varepsilon(\cdot,y), \quad g=\chi_{\Omega\setminus \overline{B_R(y)}} \cP_\varepsilon(\cdot,y).
$$
Then by Lemma \ref{180108@lem1}, there exists a unique
 $(u,p)\in Y^{1,2}_0(\Omega)^d\times \tilde{L}^2(\Omega)$ satisfying
\begin{equation}		\label{180201@B2}
\left\{
\begin{aligned}
\cL^* {u}+\nabla {p}=D_\alpha f_\alpha &\quad \text{in }\, \Omega,\\
\operatorname{div} {u}=g-(g)_\Omega &\quad \text{in }\, \Omega,\\
u=0 &\quad \text{on }\, \partial \Omega.
\end{aligned}
\right.
\end{equation}
Moreover, we have
\begin{align}
\nonumber
&\|Du\|_{L^2(\Omega)}+\|p\|_{L^2(\Omega)}\\
\nonumber
&\le C\big(\|f_\alpha\|_{L^2(\Omega)}+\|g\|_{L^2(\Omega)}\big)\\
\label{180702@A2}
&\le C\big(\|D\cG_{\varepsilon}(\cdot,y)\|_{L^2(\Omega\setminus \overline{B_R(y)})}+\|\cP_{\varepsilon}(\cdot,y)\|_{L^2(\Omega\setminus \overline{B_R(y)})}\big),
\end{align}
where $C=C(d,\lambda, K_0)$.
Since $(u,p)$ satisfies
$$
\left\{
\begin{aligned}
\cL^* {u}+\nabla {p}=0 &\quad \text{in }\, B_R(y),\\
\operatorname{div} {u}=-(g)_\Omega &\quad \text{in }\, B_{R}(y),
\end{aligned}
\right.
$$
by \eqref{180110@eq8} and H\"older's inequality, we obtain
\begin{equation}		\label{180703@eq1}
\|p\|_{L^\infty(B_{R/2}(y))}
\le C R^{-d/2}\big(\|Du\|_{L^2(\Omega)}+\|p\|_{L^2(\Omega)}\big)+C|\Omega|^{-1/2}\|g\|_{L^2(\Omega)},
\end{equation}
where $C=C(d,\lambda, \omega_A)$.
Combining \eqref{180702@A2} and \eqref{180703@eq1}, and using $R^d\le C(d)|\Omega|$,
  we see that
\begin{equation}		\label{180702@A3}
\|p\|_{L^\infty(B_{R/2}(y))}
\le C R^{-d/2}\big(\|D\cG_\varepsilon(\cdot,y)\|_{L^2(\Omega\setminus \overline{B_R(y)})}+\|\cP_\varepsilon(\cdot,y)\|_{L^2(\Omega\setminus \overline{B_R(y)})}\big),
\end{equation}
where $C=C(d,\lambda, \omega_A,K_0)$.
We apply $u$ and $\cG_{\varepsilon}(\cdot,y)$ as test functions to \eqref{180110@eq2} and \eqref{180201@B2}, respectively, to get
\begin{equation}		\label{180201@B3}
\|D \cG_\varepsilon(\cdot,y)\|_{L^2(\Omega\setminus \overline{B_R(y)})}^2+\|\cP_\varepsilon(\cdot,y)\|_{L^2(\Omega\setminus \overline{B_R(y)})}^2=\int_{B_\varepsilon(y)} p \Phi_{\varepsilon,y}\,dx.
\end{equation}
This together with \eqref{180702@A3} gives  \eqref{180618@A1}.

We now turn to the proof of \eqref{180618@A1a} when $d\ge 3$.
Let $\eta$ be a smooth function on $\bR^d$ satisfying
$$
0\le \eta\le 1, \quad \eta\equiv 1 \, \text{ on }\, B_{R/2}(y), \quad \operatorname{supp} \eta\subset B_R(y), \quad |\nabla \eta|\le C R^{-1}.
$$
Then by the Sobolev inequality, we have
$$
\begin{aligned}
&\|(1-\eta)\cG_\varepsilon(\cdot,y)\|_{L^{2d/(d-2)}(\Omega)}\\
&\le C \|D((1-\eta) \cG_\varepsilon(\cdot,y))\|_{L^{2}(\Omega)}\\
&\le C \Big(\|D \cG_\varepsilon(\cdot,y)\|_{L^2(\Omega\setminus \overline{B_{R/2}(y)})}+R^{-1}\|\cG_\varepsilon(\cdot,y)\|_{L^2(B_R(y)\setminus \overline{B_{R/2}(y)})}\Big).
\end{aligned}
$$
Notice from the Poincar\'e inequality and Lemma \ref{180110@lem1} that
$$
\begin{aligned}
&\|\cG_\varepsilon(\cdot,y)\|_{L^2(B_R(y)\setminus \overline{B_{R/2}(y)})}\\
&\le \big\|\cG_\varepsilon(\cdot,y)-(\cG_\varepsilon(\cdot,y))_{B_R(y)\setminus \overline{B_{R/2}(y)}}\big\|_{L^2(B_R(y)\setminus \overline{B_{R/2}(y)})}+CR^{-d/2}\|\cG_\varepsilon(\cdot,y)\|_{L^1(B_R(y))}\\
&\le CR\|D\cG_\varepsilon(\cdot,y)\|_{L^2(B_R(y)\setminus \overline{B_{R/2}(y)})}+CR^{-d/2+1},
\end{aligned}
$$
where $C=C(d,\lambda, \omega_A, K_0)$.
Combining these together and using \eqref{180618@A1}, we obtain
$$
\|(1-\eta)\cG_\varepsilon(\cdot,y)\|_{L^{2d/(d-2)}(\Omega)}\le C R^{-d/2},
$$
which implies  \eqref{180618@A1a} when $d\ge 3$.

Next, we prove that \eqref{180618@A1a} holds when $d=2$.
In this case, it suffices to show that
$$
\|\cG_{\varepsilon}(\cdot,y)\|_{L^\infty(\Omega_{R/16}(x))}\le CR^{-1}, \quad \forall x\in \Omega\setminus \overline{B_R(y)},
$$
where $C=C(\lambda, \omega_A, K_0, \theta)$.
We consider the following two cases:
$$
B_{R/8}(x)\cap \partial \Omega=\emptyset \quad \text{and} \quad  B_{R/8}(x)\cap \partial \Omega \neq \emptyset.
$$
Hereafter in this proof, we let $q_0=q_0( \lambda, K_0,\theta)>2$ be the constant from Lemma \ref{180723@lem1}.
\begin{enumerate}[i.]
\item
$B_{R/8}(x)\cap \partial \Omega=\emptyset$:
Let $\zeta$ be a smooth function on $\bR^2$ satisfying
$$
0\le \zeta\le 1, \quad \zeta\equiv 1 \, \text{ on }\, B_{R/16}(x), \quad \operatorname{supp} \zeta \subset B_{R/8}(x), \quad |\nabla \zeta|\le CR^{-1}.
$$
Then the pair $(v,\pi)$ given by
$$
(v,\pi)=\zeta \big(\cG_\varepsilon(\cdot,y)-(\cG_{\varepsilon}(\cdot,y))_{B_{R/8}(x)}, \cP_\varepsilon(\cdot,y)\big)
$$
satisfies
\begin{equation}		\label{180612@A4}
\left\{
\begin{aligned}
\cL v+\nabla \pi=f+D_\alpha f_\alpha &\quad \text{in }\, B_{R/8}(x),\\
\operatorname{div} v=g &\quad \text{in }\, B_{R/8}(x),\\
v=0 &\quad \text{on }\, \partial B_{R/8}(x),
\end{aligned}
\right.
\end{equation}
where
$$
f=A^{\alpha\beta}D_\beta \cG_\varepsilon(\cdot,y) D_\alpha \zeta+\nabla\zeta \cP_\varepsilon(\cdot,y),
$$
$$
f_\alpha=A^{\alpha\beta} D_\beta \zeta \big(\cG_\varepsilon(\cdot,y)-(\cG_{\varepsilon}(\cdot,y))_{B_{R/8}(x)}\big),
$$
$$
g=\nabla \zeta\cdot \big(\cG_\varepsilon(\cdot,y)-(\cG_{\varepsilon}(\cdot,y))_{B_{R/8}(x)}\big)+\zeta \big(\Phi_{\varepsilon,y}-(\Phi_{\varepsilon,y})_\Omega\big).
$$
Since it holds that
$$
|(\Phi_{\varepsilon,y})_\Omega|=|\Omega|^{-1}\le CR^{-2}
$$
and
$$
|\zeta\Phi_{\varepsilon,y}|=0 \quad \text{if }\, \varepsilon<\frac{7R}{8}, \quad |\zeta\Phi_{\varepsilon,y}|\le CR^{-2} \quad \text{if }\, \varepsilon\ge \frac{7R}{8},
$$
we obtain
$$
|g|\le CR^{-1}\big|\cG_\varepsilon(\cdot,y)-(\cG_{\varepsilon}(\cdot,y))_{B_{R/8}(x)}\big|+CR^{-2}.
$$
Using this together with
 \eqref{180618@A1}, H\"older's inequality, the Poincar\'e inequality, and the fact that
$$
\operatorname{supp} \zeta\subset B_{R/8}(x)\subset \Omega\setminus \overline{B_{R/2}(y)},
$$
one can easily check
\begin{equation}		\label{180725@A1}
\|f\|_{L^{2q_0/(2+q_0)}(B_{R/8}(x))}+\|f_\alpha\|_{L^{q_0}(B_{R/8}(x))}+\|g\|_{L^{q_0}(B_{R/8}(x))}\le CR^{2/q_0-2},
\end{equation}
where $C=C(\lambda, \omega_A, K_0, q_0)=C(\lambda, \omega_A, K_0,\theta)$.
Thus by Lemma \ref{180913@lem1} applied to \eqref{180612@A4}, we have
$$
\|Dv\|_{L^{q_0}(B_{R/8}(x))}\le CR^{2/q_0-2}.
$$
Therefore, from the Morrey inequality, Lemma \ref{180110@lem1}, and the fact that $v \in W^{1,q}_0(B_{R/8}(x))^2$, we get
$$
\|\cG_{\varepsilon}(\cdot,y)\|_{L^\infty(B_{R/16}(x))}\le \|v\|_{L^\infty(B_{R/8}(x))}+\big|(\cG_{\varepsilon}(\cdot,y))_{B_{R/16}(x)}\big|\le CR^{-1}.
$$
\item
$B_{R/8}(x)\cap \partial \Omega \neq\emptyset$:
We take a point $x_0\in \partial \Omega$ such that $d_x=|x-x_0|<R/8$.
Observe that
\begin{equation}		\label{180719@eq2a}
B_{R/8}(x)\subset B_{R/4}(x_0)\subset B_{R/2}(x_0)\subset \Omega\setminus \overline{B_{3R/8}(y)}.
\end{equation}
Let $\tilde{\zeta}$ be a smooth function on $\bR^2$ satisfying
$$
0\le \tilde{\zeta} \le 1, \quad \tilde{\zeta} \equiv 1 \, \text{ on }\, B_{R/4}(x_0), \quad \operatorname{supp} \tilde{\zeta} \subset B_{R/2}(x_0), \quad |\nabla \tilde{\zeta} |\le CR^{-1}.
$$
Then the pair $(\tilde{v}, \tilde{\pi})$ given by
$$
(\tilde{v},\tilde{\pi}):=\tilde{\zeta} \big(\cG_\varepsilon(\cdot,y), \cP_\varepsilon(\cdot,y)\big)
$$
satisfies
\begin{equation}		\label{180719@eq2}
\left\{
\begin{aligned}
\cL \tilde{v}+\nabla \tilde{\pi}=\tilde{f}+D_\alpha \tilde{f}_\alpha &\quad \text{in }\, \Omega,\\
\operatorname{div} \tilde{v}=\tilde{g} &\quad \text{in }\, \Omega,\\
v=0 &\quad \text{on }\, \partial \Omega,
\end{aligned}
\right.
\end{equation}
where
$$
\tilde{f}=A^{\alpha\beta}D_\beta \cG_\varepsilon(\cdot,y) D_\alpha \tilde{\zeta}+\nabla\tilde{\zeta} \cP_\varepsilon(\cdot,y),
 \quad \tilde{f}_\alpha=A^{\alpha\beta} D_\beta \tilde{\zeta} \cG_\varepsilon(\cdot,y),
$$
$$
\tilde{g}=\nabla \tilde{\zeta}\cdot \cG_\varepsilon(\cdot,y)+\tilde{\zeta} \big(\Phi_{\varepsilon,y}-(\Phi_{\varepsilon,y})_\Omega\big).
$$
Similar to \eqref{180725@A1},
by using \eqref{180618@A1}, \eqref{180719@eq2a}, H\"older's inequality, and the following boundary Poincar\'e inequality
$$
\|\cG_{\varepsilon}(\cdot,y)\|_{L^{q_0}(\Omega_{R/2}(x_0))}\le C\|D\cG_{\varepsilon}(\cdot,y)\|_{L^{2q_0/(2+q_0)}(\Omega_{R/2}(x_0))},
$$
we have
\begin{equation}		\label{180725@C1}
\|\tilde{f}\|_{L^{2q_0/(2+q_0)}(\Omega)}+\|\tilde{f}_\alpha\|_{L^{q_0}(\Omega)}+\|\tilde{g}\|_{L^{q_0}(\Omega)}\le CR^{2/q_0-2},
\end{equation}
where $C=C(\lambda, \omega_A, K_0,\theta)$.
Thus by Lemma \ref{180913@lem1} applied to \eqref{180719@eq2}, we have
$$
\|D\tilde{v}\|_{L^{q_0}(\Omega_{R/2}(x_0))}\le CR^{2/q_0-2}.
$$
Therefore, we get from  the Morrey inequality and $\tilde{v}(x_0)=0$ that
$$
\|\cG_{\varepsilon}(\cdot,y)\|_{L^\infty(\Omega_{R/8}(x))}\le \|\tilde{v}\|_{L^\infty(\Omega_{R/2}(x_0))}\le CR^{-1}.
$$
\end{enumerate}
The lemma is proved.
\end{proof}

From Lemma \ref{180719@lem1}, we obtain the following uniform weak type estimates.

\begin{lemma}		\label{180719@lem2}
Let $y\in \Omega$ and $\varepsilon\in (0, 1]$.
Then we have
$$
\big|\{x\in \Omega:|\cG_\varepsilon(x,y)|>c_0t\}\big|\le C t^{-d/(d-1)}, \quad t\ge (d_y^*)^{1-d}
$$
and
$$
\big|\{x\in \Omega: |D_x\cG_\varepsilon(x,y)|>t\}\big|+\big|\{x\in \Omega:|\cP_\varepsilon(x,y)|>t\}\big|\le C' t^{-1}, \quad t\ge (d_y^*)^{-d},
$$
where $c_0=1$ if $d\ge 3$ and $c_0=c_0(\lambda, \omega_A, K_0,\theta)>0$ if $d=2$.
In the above,
$(C,C')=(C,C')(d,\lambda,\omega_A, K_0)$ and $C$ depends also on $\theta$ when $d=2$.
\end{lemma}

\begin{proof}
We only prove the first inequality because the other is the same with obvious modifications.
Let us set
\begin{equation}		\label{180719@eq4}
\cA_t=\{x\in \Omega:|\cG_{\varepsilon}(x,y)|>c_0t\},
\end{equation}
where $c_0$ is a constant to be chosen below.
We consider the following two cases:
$$
d\ge 3 \quad \text{and}\quad d=2.
$$
\begin{enumerate}[i.]
\item
$d\ge 3$:
Let $c_0=1$ and  $t\ge (d_y^*)^{1-d}$.
Then by \eqref{180618@A1a} with $R=t^{-1/(d-1)}$, we have
$$
\begin{aligned}
|\cA_t|&=|\cA_t\cap B_R(y)|+|\cA_t\setminus B_R(y)|\\
&\le CR^d+ \frac{1}{t^{2d/(d-2)}}\int_{\cA_t\setminus B_R(y)} |\cG_{\varepsilon}(\cdot,y)|^{2d/(d-2)}\,dx\le C t^{-d/(d-1)}.
\end{aligned}
$$
\item
$d=2$: Let $t\ge (d_y^*)^{-1}$.
Then by \eqref{180618@A1a} with $R=t^{-1}$, there exists a constant $c_0=c_0(\lambda, \omega_A, K_0,\theta)>0$ such that
$$
\|\cG_{\varepsilon}(\cdot,y)\|_{L^\infty(\Omega\setminus \overline{B_{R}(y)})}\le c_0R^{-1}.
$$
Therefore, we have
$$
|\cA_t|=|\cA_t\cap B_R(y)|\le CR^2=Ct^{-2}.
$$
\end{enumerate}
The lemma is proved.
\end{proof}

The following lemma is a simple consequence of Lemma \ref{180719@lem2}.

\begin{lemma}		\label{180719@lem3}
Let $y\in \Omega$, $\varepsilon\in (0,1]$, and $R\in (0, d_y^*]$.
Then we have
$$
\|\cG_{\varepsilon}(\cdot,y)\|_{L^q(B_R(y))}\le C R^{1-d+d/q}, \quad 0<q<\frac{d}{d-1},
$$
and
$$
\|D\cG_{\varepsilon}(\cdot,y)\|_{L^q(B_R(y))}+\|\cP_{\varepsilon}(\cdot,y)\|_{L^q(B_R(y))}\le C' R^{-d+d/q}, \quad 0<q<1,
$$
where $(C,C')=(C,C')(d,\lambda,\omega_A, K_0)$ and $C$ depends also on $\theta$ when $d=2$.
\end{lemma}

\begin{proof}
We only prove the first inequality because the other is the same with obvious modifications.
Let $q\in (0, d/(d-1))$ and recall the notation \eqref{180719@eq4}.
Then by the first inequality with $t=R^{-(d-1)}$ in Lemma \ref{180719@lem2}, we have
$$
\begin{aligned}
\int_{B_R(y)}|\cG_{\varepsilon}(\cdot,y)|^q\,dx
&=q\int_0^\infty s^{q-1}\big|\{x\in B_R(y):|\cG_{\varepsilon}(x,y)|>s\}\big|\,ds\\
&=c_0q\int_0^\infty (c_0s)^{q-1}\big|\{x\in B_R(y):|\cG_{\varepsilon}(x,y)|>c_0s\}\big|\,ds\\
&\le C\int_0^{t} s^{q-1}|B_R(y)|\,ds+C\int_{t}^\infty s^{q-1}|\cA_s|\,ds\\
&\le CR^{q-dq+d},
\end{aligned}
$$
which gives the desired estimate.
\end{proof}

%========================================
\section{Proofs of main theorems}			\label{S_Pr}
%========================================

This section is devoted to the proofs of the theorems in Section \ref{S3}.
Throughout this section, we denote by $(\cG_{\varepsilon}(\cdot,y), \cP_{\varepsilon}(\cdot,y))$ the approximated Green function constructed in Section \ref{S_Ap}.

%========================================
\subsection{Proof of Theorem \ref{M1}}	
%========================================
By Lemmas \ref{180719@lem1} and \ref{180719@lem3}, the weak compactness theorem, and a diagonalization process, we easily see that there exist a sequence $\{\varepsilon_\rho\}_{\rho=1}^\infty$ tending to zero and a pair $(\cG(\cdot,y),\cP(\cdot,y))$ such that for any $R \in (0,  d_y^*]$,
\begin{equation}		\label{180121@eq2}
\left.
\begin{aligned}
(1-\eta_R)\cG_{\varepsilon_\rho}(\cdot,y) \rightharpoonup (1-\eta_R)\cG(\cdot,y) & \quad \text{weakly in }\, Y^{1,2}_0(\Omega)^d,\\
\cP_{\varepsilon_\rho}(\cdot,y) \rightharpoonup \cP(\cdot,y) &\quad \text{weakly in }\, L^2(\Omega\setminus \overline{B_R(y)}),
\end{aligned}
\right.
\end{equation}
where $\eta_R$ is any smooth function in $\bR^d$ satisfying $\eta_R\equiv 1$ on $B_R(y)$, and that for fixed $q\in \big(1,\frac{d}{d-1}\big)$,
\begin{equation}		\label{180121@eq1}
\cG_{\varepsilon_\rho}(\cdot,y) \rightharpoonup \cG(\cdot,y) \quad \text{weakly in }\, L^q(B_{d_y^*}(y))^d.
\end{equation}
Moreover, we obtain the following uniform convergence.

\begin{lemma}		\label{180223@lem1}
For given compact set $\cK\subset \Omega\setminus \{y\}$, there is a subsequence of $\{\cG_{\varepsilon_\rho}(\cdot,y)\}$ that converges to $\cG(\cdot,y)$ uniformly on $\cK$.
\end{lemma}

\begin{proof}
Let ${B_R(x)}\subset \Omega\setminus \{y\}$ and $0<\varepsilon<\min\{d_y^*, \operatorname{dist}(y, B_R(x))\}$.
Note that
$$
\left\{
\begin{aligned}
\cL \cG_\varepsilon(\cdot,y) +\nabla \cP_\varepsilon(\cdot,y)=0 &\quad \text{in }\, B_R(x),\\
\operatorname{div} \cG_\varepsilon(\cdot,y)=-|\Omega|^{-1} &\quad \text{in }\,B_R(x).
\end{aligned}
\right.
$$
By Lemmas \ref{180719@lem1} and \ref{180108@lem3}, we see that $\|\cG_\varepsilon(\cdot,y)\|_{W^{1,\infty}(B_{R/2}(x))}\le C$,
where $C$ does not depend on $\varepsilon$.
This implies that $\{\cG_{\varepsilon}(\cdot,y)\}$ is equicontinuous and uniformly bounded on $\overline{B_{R/2}(x)}$.
Therefore, we get the desired conclusion from the Arzel\'a-Ascoli theorem.
\end{proof}

The pair $(\cG, \cP)$ satisfies the estimates in Remark \ref{180217@rmk5}.
Indeed, the estimates in the assertion $(a)$ are simple consequences of Lemma \ref{180719@lem1} and the weak lower semi-continuity.
Then by following the same steps used in Lemmas \ref{180719@lem2} and \ref{180719@lem3}, we see that $(\cG, \cP)$ satisfies the estimates in the assertions $(b)$ and $(c)$.

Now we shall show that $(\cG, \cP)$ satisfies the properties $(a)$--$(c)$ in Definition \ref{D0} so that it is a Green function for the pressure of $\cL$ in $\Omega$.

Obviously, $(\cG, \cP)$ satisfies the property $(a)$.
To verify the property $(b)$, we apply $\phi\in C^\infty_0(\Omega\setminus \overline{B_r(y)})^d$ as a test function to \eqref{180110@eq2} and use \eqref{180121@eq2} to get
\begin{align*}
&\int_{\Omega\setminus \overline{B_r(y)}} A^{\alpha\beta}D_\beta \cG(\cdot,y)\cdot D_\alpha \phi\,dx+\int_{\Omega\setminus \overline{B_r(y)}} \cP(\cdot,y) \operatorname{div}\phi\,dx\\
&=\lim_{\rho\to \infty}\bigg(\int_{\Omega\setminus \overline{B_r(y)}} A^{\alpha\beta}D_\beta \cG_{\varepsilon_\rho}(\cdot,y)\cdot D_\alpha \phi\,dx+\int_{\Omega\setminus \overline{B_r(y)}} \cP_{\varepsilon_\rho}(\cdot,y) \operatorname{div}\phi\,dx\bigg)=0.
\end{align*}
This implies that
$$
\cL \cG(\cdot,y)+\nabla \cP(\cdot,y)=0 \quad \text{in }\,  \Omega\setminus \overline{B_r(y)}.
$$
Similarly, by applying $\varphi\in C^\infty_0(\Omega)$ as a test function to the divergence equation in \eqref{180110@eq2}, and using \eqref{180121@eq2} and \eqref{180121@eq1}, we have
\begin{align*}
&\int_{\Omega}\cG(\cdot,y) \cdot \nabla \varphi\,dx\\
&=\int_{B_{d_y^*}(y)}\cG(\cdot,y)\cdot \nabla(\zeta \varphi)\,dx+\int_{\Omega\setminus \overline{B_{d_y^*/2}(y)}}\cG(\cdot,y)\cdot \nabla ((1-\zeta)\varphi)\,dx\\
&=-\lim_{\rho\to \infty}\bigg(\int_{\Omega} \operatorname{div}\cG_{\varepsilon_\rho}(\cdot,y)  \zeta \varphi\,dx+\int_{\Omega}\operatorname{div}\cG_{\varepsilon_\rho}(\cdot,y)(1-\zeta)\varphi\,dx\bigg)\\
&=-\lim_{\rho\to \infty}\bigg(\int_{\Omega} \Phi_{\varepsilon_\rho,y}   \varphi\,dx-(\Phi_{\varepsilon_\rho,y})_\Omega \int_{\Omega}\varphi\,dx\bigg)\\
&=-\varphi(y)+\dashint_{\Omega} \varphi \,dx,
\end{align*}
where $\zeta$ is a smooth function on $\bR^d$ satisfying $\zeta\equiv 1$ on $B_{d_y^*/2}(y)$ and $\operatorname{supp}\zeta\subset B_{d_y^*}(y)$.
This implies that
$$
\operatorname{div} \cG(\cdot,y)=\delta_y-\frac{1}{|\Omega|} \quad \text{in }\, \Omega,
$$
and thus  $(\cG, \cP)$ satisfies the property $(b)$.
Therefore, by applying Lemma \ref{180108@lem3} to \eqref{180212@A1},
$\cG(\cdot,y)$ is continuous in $\Omega\setminus \{y\}$ and
$$
(\cG(\cdot,y), \cP(\cdot,y))\in W^{1,\infty}_{\rm{loc}}(\Omega\setminus \{y\})^d\times L^\infty_{\rm{loc}}(\Omega\setminus \{y\}).
$$
To verify the property $(c)$ in Definition \ref{D0}, suppose that $(u,p)\in Y^{1,2}_0(\Omega)^d\times \tilde{L}^2(\Omega)$ is a unique weak solution of \eqref{180206@eq2}.
By testing \eqref{180206@eq2} and \eqref{180110@eq2} with $\cG_{\varepsilon_\rho}(\cdot,y)$ and $u$, respectively, we have
\begin{equation}		\label{180206@eq3}
\int_{\Omega} p \Phi_{\varepsilon_\rho,y}\,dx=-\int_\Omega \cG_{\varepsilon_\rho}(x,y)\cdot f(x)\,dx+\int_\Omega \cP_{\varepsilon_\rho}(x,y) g(x)\,dx.
\end{equation}
Notice from \eqref{180121@eq2} and \eqref{180121@eq1} that for any $y\in \Omega\setminus \overline{\operatorname{supp}g}$, the right-hand side of \eqref{180206@eq3} converges to
$$
-\int_\Omega \cG(x,y)\cdot f(x)\,dx+\int_\Omega \cP(x,y) g(x)\,dx\quad \text{as }\, \rho\to \infty.
$$
On the other hand, the left-hand side of \eqref{180206@eq3} converges to $p(y)$ for any $y$ in the Lebesgue set of $p$.
This implies that $(\cG,\cP)$ satisfies the property $(c)$ in Definition \ref{D0}.

To complete the proof of the theorem, it remains to show the estimates \eqref{180207@eq1a} and \eqref{180302@eq1}.
Let $x,y\in \Omega$ with $0<|x-y|\le d_y^*/2$, and set $R=|x-y|/2$.
Since $(\cG(\cdot,y), \cP(\cdot,y))$ satisfies
$$
\left\{
\begin{aligned}
\cL \cG(\cdot,y)+\nabla \cP(\cdot,y)=0 &\quad \text{in }\, B_R(x),\\
\operatorname{div} \cG(\cdot,y)=-\frac{1}{|\Omega|} &\quad \text{in }\, B_R(x),
\end{aligned}
\right.
$$
we  obtain by \eqref{180110@eq8} and $R^d\le C |\Omega|$ that
$$
\begin{aligned}
&\|D\cG(\cdot,y)\|_{L^\infty(B_{R/2}(x))}+\|\cP(\cdot,y)\|_{L^\infty(B_{R/2}(x))}\\
&\le C R^{-d}\big(\|D\cG(\cdot,y)\|_{L^1(B_R(x))}+\|\cP(\cdot,y)\|_{L^1(B_R(x))}\big)+CR^{-d},
\end{aligned}
$$
where $C=C(d,\lambda,\omega_A)$.
Thus from H\"older's inequality, $B_R(x)\subset \Omega\setminus B_R(y)$, and Remark \ref{180217@rmk5} $(a)$, we get
$$
\begin{aligned}
&\|D\cG(\cdot,y)\|_{L^\infty(B_{R/2}(x))}+\|\cP(\cdot,y)\|_{L^\infty(B_{R/2}(x))}\\
&\le C R^{-d/2}\Big(\|D\cG(\cdot,y)\|_{L^2(\Omega\setminus \overline{B_R(y)})}+\|\cP(\cdot,y)\|_{L^2(\Omega\setminus \overline{B_R(y)})}\Big)+ CR^{-d}\\
&\le CR^{-d},
\end{aligned}
$$
where $C=C(d,\lambda,\omega_A,K_0)$.
This implies \eqref{180302@eq1}.
Similarly, by \eqref{180110@eq8a} and Remark \ref{180217@rmk5} $(a)$, we have that
$$
\begin{aligned}
&\|\cG(\cdot,y)\|_{L^\infty(B_{R/2}(x))}\\
&\le CR^{-d}\|\cG(\cdot,y)\|_{L^1(B_{R/2}(x))}\\
&\quad +CR^{1-d}\big(\|D\cG(\cdot,y)\|_{L^1(B_R(x))}+\|\cP(\cdot,y)\|_{L^1(B_R(x))}\big)+CR^{1-d}\\
&\le C R^{1-d/2} \|\cG(\cdot,y)\|_{L^{q^*}(\Omega\setminus \overline{B_R(y)})}\\
&\quad +C R^{1-d/2}\Big(\|D\cG(\cdot,y)\|_{L^2(\Omega\setminus \overline{B_R(y)})}+\|\cP(\cdot,y)\|_{L^2(\Omega\setminus \overline{B_R(y)})}\Big)+CR^{1-d}\\
&\le CR^{1-d},
\end{aligned}
$$
where $C=C(d,\lambda, \omega_A, K_0)$ and $C$ depends also on $\theta$ when $d=2$. This together with the continuity of $\cG(\cdot,y)$ in $B_{R/2}(x)$ gives \eqref{180207@eq1a}. The theorem is proved.
\qed

%========================================
\subsection{Proof of Theorem \ref{M2}}	
%========================================
Let $(\cG, \cP)$ be the Green function for the pressure of $\cL$ constructed in Theorem \ref{M1}.
Obviously, by Lemma \ref{180208@lem1} and Definition \ref{D0} $(b)$, we see that
$$
(\cG(\cdot,y), \cP(\cdot,y))\in C^1(\overline{\Omega\setminus B_R(y)})^d\times C(\overline{\Omega\setminus B_R(y)})
$$
for all $y\in \Omega$ and $R>0$.

We divide the proof into several steps.

{{\bf{\em{Step 1.}}}
In this step, we establish various boundary estimates for the approximated Green function $(\cG_\varepsilon, \cP_\varepsilon)$.
The following lemma is about the $L^1$-estimate for $\cG_\varepsilon$, which is the  counterpart of  Lemma \ref{180110@lem1}.

\begin{lemma}		\label{180216@lem1}
Let $y\in \Omega$ and  $\varepsilon\in (0,1]$.
Then for any $x\in \Omega$ and $R\in (0,1]$, we have
$$
\|\cG_\varepsilon(\cdot,y)\|_{L^1(\Omega_R(x))}\le CR,
$$
where $
C=C(d,\lambda,\omega_A,K_0,  R_0,\varrho_0)$ and $C$ depends also on $\theta$ when $d=2$.
\end{lemma}

\begin{proof}
Due to the interior estimate in Lemma \ref{180110@lem1}, it suffices to consider the case when $d_x<R\le 1$.
Moreover, because the proof of Lemma \ref{180110@lem1} still works for $2\varepsilon>R$, we only need to prove the lemma with $2\varepsilon\le R$.

Now, we assume that $d_x<R\le 1$ and $2\varepsilon\le R$.
Let $(u,p)\in W^{1,2}_0(\Omega)^d\times \tilde{L}^2(\Omega)$ be a unique weak solution of the problem
$$
\left\{
\begin{aligned}
\cL^* u+\nabla p=-f &\quad \text{in }\, \Omega,\\
\operatorname{div} u=0 &\quad \text{in }\, \Omega,\\
u=0 &\quad \text{on }\, \partial \Omega,
\end{aligned}
\right.
$$
where $f=\chi_{\Omega_{R}(x)}\big(\operatorname{sgn}\cG_\varepsilon^1(\cdot,y), \ldots, \operatorname{sgn}\cG_\varepsilon^d(\cdot,y)\big)^{\top}$.
Then by using \eqref{180304@eq1} and following the same argument used in deriving \eqref{180726@eq1}, we have
\begin{equation}		\label{180914@eq1}
\begin{aligned}
\|\cG_\varepsilon(\cdot,y)\|_{L^1(\Omega_{R}(x))}
&\le CR^{-d-1}\|u\|_{L^1(\Omega_R(y))}\\
&\quad +C R^{-d}\big(\|Du\|_{L^1(\Omega_R(y))}+\|p\|_{L^1(\Omega_R(y))}\big)+CR,
\end{aligned}
\end{equation}
where $C=C(d,\lambda, \omega_A, R_0, \varrho_0)$.
If $d\ge 3$, then by Lemma \ref{180108@lem1} and the Sobolev inequality, we have
$$
\|u\|_{L^{2d/(d-2)}(\Omega)}+\|Du\|_{L^2(\Omega)}+\|p\|_{L^2(\Omega)}\le CR^{(d+2)/2}.
$$
From this together with  \eqref{180914@eq1} and H\"older's inequality, we get the desired estimate.
On the other hand, if $d=2$, then by using the Morrey inequality and the fact that $\cG_{\varepsilon}(\cdot,y)=0$ on $\partial \Omega\cap B_R(x)$, we have
$$
\|u\|_{L^\infty(\Omega_R(x))}\le CR^{1-2/q_{0}}\|Du\|_{L^{q_0}(\Omega_R(x))},
$$
where $q_0=q_0(\lambda,K_0,\theta)>2$ is the constant from Lemma \ref{180723@lem1}.
From this together with Lemma \ref{180913@lem1} and Remark \ref{180717@rmk1}, we get
$$
R^{-1+2/q_{0}}\|u\|_{L^\infty(\Omega_R(x))}+\|Du\|_{L^{q_{0}}
(\Omega_R(x))}+\|p\|_{L^{q_{0}}(\Omega_R(x))}\le CR^{(2+q_{0})/q_{0}},
$$
where $C=C(\lambda, K_0,\theta)$.
Therefore, by \eqref{180914@eq1} and H\"older's inequality, we obtain the desired estimate.
\end{proof}

The following lemma is an analog of Lemma \ref{180719@lem1}.

\begin{lemma}		\label{180324@lem1}
Let $y\in \Omega$, $\varepsilon\in (0,1]$, $R\in (0, 1]$, and
$$
q^*=\frac{2d}{d-2} \quad \text{if }\, d\ge 3, \quad q^*=\infty \quad \text{if }\, d=2.
$$
Then we have
\begin{equation}		\label{180702@eq2}
\|\cG_{\varepsilon}(\cdot,y)\|_{L^{q^*}(\Omega\setminus \overline{B_R(y)})}\le CR^{-d/2}
\end{equation}
and
\begin{equation}		\label{180702@eq1}
\|DG_{\varepsilon}(\cdot,y)\|_{L^2(\Omega\setminus \overline{B_R(y)})}+\|\cP_{\varepsilon}(\cdot,y)\|_{L^2(\Omega\setminus \overline{B_R(y)})}\le C'R^{-d/2},
\end{equation}
where $(C,C')=(C,C')(d,\lambda, \omega_A, K_0, R_0, \varrho_0)$ and $C$ depends also on $\theta$ when $d=2$.
\end{lemma}

\begin{proof}
By utilizing Lemma \ref{180216@lem1} and \eqref{180702@eq1}, and following the same steps as in the proof of \eqref{180618@A1a}, one can show that \eqref{180702@eq2} holds.
Thus we only prove \eqref{180702@eq1}.
Due to \eqref{180618@A1} in Lemma \ref{180719@lem1}, it suffices to consider the case of  $d_y<R\le 1$.

Let $y\in \Omega$, $\varepsilon\in (0,1]$, and $d_y<R\le 1$.
If $\varepsilon>R/2$, then \eqref{180702@eq1} follows immediately from \eqref{180110@eq4}.
Now we assume $\varepsilon\le R/2$.
Set
$$
f_\alpha=\chi_{\Omega\setminus \overline{B_R(y)}} D_\alpha \cG_\varepsilon(\cdot,y), \quad g=\chi_{\Omega\setminus \overline{B_R(y)}} \cP_\varepsilon(\cdot,y).
$$
Then by Lemma \ref{180108@lem1}, there exists a unique solution $(u,p)\in Y^{1,2}_0(\Omega)^d\times \tilde{L}^2(\Omega)$ of the problem
$$
\left\{
\begin{aligned}
\cL^* {u}+\nabla {p}=D_\alpha f_\alpha &\quad \text{in }\, \Omega,\\
\operatorname{div} {u}=g-(g)_\Omega &\quad \text{in }\, \Omega,\\
u=0 &\quad \text{on }\, \partial \Omega,
\end{aligned}
\right.
$$
satisfying
\begin{equation}		\label{180717@A1}
\|Du\|_{L^2(\Omega)}+\|p\|_{L^2(\Omega)}
\le C\Big(\|D\cG_{\varepsilon}(\cdot,y)\|_{L^2(\Omega\setminus \overline{B_R(y)})}+\|\cP_{\varepsilon}(\cdot,y)\|_{L^2(\Omega\setminus \overline{B_R(y)})}\Big),
\end{equation}
where $C=C(d,\lambda, K_0)$.
Moreover, it follows from \eqref{180304@eq1} that (see \eqref{180703@eq1})
$$
\begin{aligned}		
\|p\|_{L^\infty(\Omega_{R/2}(y))}
&\le C R^{-d/2-1}\|u\|_{L^2(\Omega_R(y))}\\
&\quad +CR^{-d/2}\big(\|Du\|_{L^2(\Omega_R(y))}+\|p\|_{L^2(\Omega_R(y))}\big)
+C|\Omega|^{-1/2}\|g\|_{L^2(\Omega)},
\end{aligned}
$$
where $C=C(d,\lambda, \omega_A, R_0, \varrho_0)$. Hence by using the fact that
\begin{equation}		\label{180915@eq1}
R^d\le C(d, R_0, \varrho_0)|\Omega|,
\end{equation}
we have
\begin{equation}		\label{180717@A1a}
\begin{aligned}		
\|p\|_{L^\infty(\Omega_{R/2}(y))}
&\le C R^{-d/2-1}\|u\|_{L^2(\Omega_R(y))}\\
&\quad +CR^{-d/2}\big(\|Du\|_{L^2(\Omega_R(y))}+\|p\|_{L^2(\Omega_R(y))}+\|g\|_{L^2(\Omega)}\big).
\end{aligned}
\end{equation}
Observe that
\begin{equation}		\label{180703@eq2b}
\|u\|_{L^2(\Omega_{2R}(y))}\le CR\|Du\|_{L^2(\Omega_{2R}(y))},
\end{equation}
where $C=C(d,R_0, \varrho_0)$.
Indeed, if we take a point $y_0\in \partial \Omega$ satisfying $d_y=|y-y_0|<R$, then by \eqref{180703@eq5} and $(B_R(y_0)\setminus \Omega) \subset (B_{2R}(y)\setminus \Omega)$, we have
$$
|B_{2R}(y)\setminus \Omega|\ge |B_R(y_0)\setminus \Omega|\ge \theta R^d,
$$
where $\theta=\theta(d,R_0, \varrho_0)$.
This together with the boundary Poincar\'e inequality (see, for instance, \cite[Eq. (7.45)]{MR1814364}) gives \eqref{180703@eq2b}.
Combining \eqref{180717@A1}, \eqref{180717@A1a}, and \eqref{180703@eq2b}, we have
$$
\|p\|_{L^\infty(\Omega_{R/2}(y))}\le CR^{-d/2}\Big(\|D\cG_{\varepsilon}(\cdot,y)\|_{L^2(\Omega\setminus \overline{B_R(y)})}+\|\cP_{\varepsilon}(\cdot,y)\|_{L^2(\Omega\setminus \overline{B_R(y)})}\Big),
$$
where $C=C(d,\lambda, \omega_A, K_0, R_0, \varrho_0)$.
Therefore, by using the identity (see \eqref{180201@B3})
$$
\|D \cG_\varepsilon(\cdot,y)\|_{L^2(\Omega\setminus \overline{B_R(y)})}^2+\|\cP_\varepsilon(\cdot,y)\|_{L^2(\Omega\setminus \overline{B_R(y)})}^2=\int_{\Omega_\varepsilon(y)} p \Phi_{\varepsilon,y}\,dx,
$$
we conclude \eqref{180702@eq1}.
The lemma is proved.
\end{proof}

{{\bf{\em{Step 2.}}}
In this step, we prove the estimates in the theorem and Remarks \ref{180326@rmk1} and \ref{180218@rmk1}.
The estimates in the assertion $(a)$ in  Remark \ref{180218@rmk1} are simple consequences of  Lemma \ref{180324@lem1} and the weak semi-continuity.
Then by following the same steps used in Lemmas \ref{180719@lem2} and \ref{180719@lem3}, one can easily obtain the estimates in the assertions $(b)$ and $(c)$ in Remark \ref{180218@rmk1}.

To prove \eqref{180208@eq1} and \eqref{180217@C1},
let $x,y\in \Omega$ with $0<|x-y|\le 1$, and set $R=|x-y|/2$.
Since $(\cG(\cdot,y), \cP(\cdot,y))$ satisfies
$$
\left\{
\begin{aligned}
\cL \cG(\cdot,y)+\nabla \cP(\cdot,y)=0 &\quad \text{in }\, \Omega_R(x),\\
\operatorname{div} \cG(\cdot,y)=-\frac{1}{|\Omega|} &\quad \text{in }\, \Omega_R(x),\\
\cG(\cdot,y)=0 &\quad \text{on }\, \partial \Omega \cap B_R(x),
\end{aligned}
\right.
$$
by \eqref{180915@eq1}, \eqref{180304@eq1}, H\"older's inequality, and the fact that $B_R(x)\subset \Omega\setminus B_R(y)$, we have
\begin{align*}
&R^{-1}\|\cG(\cdot,y)\|_{L^\infty(\Omega_{R/2}(x))}+\|D\cG(\cdot,y)\|_{L^\infty(\Omega_{R/2}(x))}+\|\cP(\cdot,y)\|_{L^\infty(\Omega_{R/2}(x))}\\
&\le C R^{-d/2}\|\cG(\cdot,y)\|_{L^{q^*}(\Omega\setminus \overline{B_R(y)})}\\
&\quad +C R^{-d/2}\Big(\|D\cG(\cdot,y)\|_{L^{2}(\Omega\setminus \overline{B_R(y)})}+\|\cP(\cdot,y)\|_{L^2(\Omega\setminus \overline{B_R(x)})}\Big)+CR^{-d},
\end{align*}
where $q^*=2d/(d-2)$ if $d\ge 3$ and $q^*=\infty$ if $d=2$.
Here, the constant $C$ depends only on $d$, $\lambda$, $\omega_A$, $K_0$, $R_0$,  and $\varrho_0$.
From this inequality and Remark \ref{180218@rmk1} $(a)$, we get
\begin{equation}		\label{180705@eq1}
\begin{aligned}
&R^{-1}\|\cG(\cdot,y)\|_{L^\infty(\Omega_{R/2}(x))}+\|D\cG(\cdot,y)\|_{L^\infty(\Omega_{R/2}(x))}\\
&\quad +\|\cP(\cdot,y)\|_{L^\infty(\Omega_{R/2}(x))}\le CR^{-d},
\end{aligned}
\end{equation}
where $C=C(d,\lambda, \omega_A, K_0, R_0, \varrho_0)$ and $C$ depends also on $|\Omega|$ when $d=2$.
Therefore, by the continuity of $\cG(\cdot,y)$, $D\cG(\cdot,y)$, and $\cP(\cdot,y)$,  we conclude \eqref{180208@eq1} and \eqref{180217@C1}.

We now turn to the proof of the estimate in Remark \ref{180326@rmk1}.
Let $x,y\in \Omega$ with $0<|x-y|\le 1$, and set $R=|x-y|/2$.
We assume $d_x<R/2$ and extend $\cG(\cdot,y)$ by zero on $\bR^d\setminus \Omega$.
Then by taking $x_0\in \partial \Omega$ such that $d_x=|x-x_0|$, and using \eqref{180705@eq1} and  $\cG(x_0, y)=0$, we have
$$
|\cG(x,y)|=|\cG(x,y)-\cG(x_0,y)|\le C d_x \|D\cG(\cdot,y)\|_{L^\infty(B_{R/2}(x))}\le Cd_x R^{-d}.
$$
From this together with \eqref{180208@eq1}, we get
$$
|\cG(x,y)| \le C\min \{d_x, |x-y|/4\}\cdot |x-y|^{-d} ,
$$
which yields  the estimate in Remark \ref{180326@rmk1}.

{{\bf{\em{Step 3.}}}
In this step, we prove that  \eqref{180220@eq3} holds.
Let $(\cG^*, \cP^*)$ be the Green function for the pressure of the adjoint operator $\cL^*$ and $(\cG^*_{\sigma}, \cP^*_\sigma)$ be its approximated Green function.
More precisely, for given $x\in \Omega$ and $\sigma\in (0,1]$, $(\cG^*_\sigma(\cdot,x), \cP^*_\sigma(\cdot,x))\in Y^{1,2}_0(\Omega)^d\times \tilde{L}^2(\Omega)$ is a unique weak solution of the problem
\begin{equation}		\label{180222@E1}
\left\{
\begin{aligned}
\cL^* \cG_\sigma^*(\cdot,x)+\nabla \cP_\sigma^*(\cdot,x)=0 &\quad \text{in }\, \Omega,\\
\operatorname{div} \cG_\sigma^*(\cdot,x)=\Phi_{\sigma,x}-(\Phi_{\sigma,x})_\Omega & \quad \text{in }\, \Omega,\\
\cG^*_{\sigma}(\cdot,x)=0 &\quad \text{on }\, \partial \Omega,
\end{aligned}
\right.
\end{equation}
where $\Phi_{\sigma,x}$ is as in Section \ref{S_Ap}.
By proceeding similarly as in the proof of Theorem \ref{M1}, there exists a sequence $\{\sigma_\tau\}_{\tau=1}^\infty$ satisfying the natural counterparts of \eqref{180121@eq2} and \eqref{180121@eq1}.

We first claim that
\begin{equation}		\label{180220@eq4a}
\lim_{\rho\to \infty}\cP_{\varepsilon_\rho}(x,y)=\cP^*(y,x) \quad \text{for all }\, x,y\in \Omega \, \text{ with }\, x\neq y.
\end{equation}
We apply $\cG_{\sigma_\tau}^*(\cdot,x)$ and $\cG_{\varepsilon_\rho}(\cdot,y)$ as test functions to  \eqref{180110@eq2} and \eqref{180222@E1}, respectively, to get
$$
\int_\Omega \cP_{\varepsilon_\rho}(z ,y) \Phi_{\sigma_\tau,x}(z)\,dz=\int_\Omega \cP_{\sigma_\tau}^*(z,x) \Phi_{\varepsilon_\rho, y}(z)\,dz \quad \text{for any }\, x,y\in \Omega.
$$
By the continuity of $\cP_{\varepsilon_\rho}(\cdot,y)$, the left-hand side of the above inequality converges to $\cP_{\varepsilon_\rho}(x,y)$ as $\tau \to \infty$.
On the other hand, by the counterpart of \eqref{180121@eq2}, the right-hand side converges to
$$
\int_\Omega \cP^*(z,x) \Phi_{\varepsilon_\rho,y}(z)\,dz \quad \text{as }\, \tau \to \infty
$$
if $x\neq y$ and $0<\varepsilon_\rho\le \min\{1,|x-y|/2\}$.
Combining these together, we have
$$
\cP_{\varepsilon_{\rho}}(x,y)=\int_\Omega \cP^*(z,x)\Phi_{\varepsilon_\rho,y}(z)\,dz \quad \text{if }\, 0<\varepsilon_{\rho}\le \min\{1,|x-y|/2\},
$$
and thus, from the continuity of $\cP^*(\cdot,x)$ on $\Omega\setminus \{x\}$, we get the claim \eqref{180220@eq4a}.

Next, we claim that
\begin{equation}		\label{180325@eq2}
\cP^*(y,\cdot)\in L^1_{\rm{loc}}(\overline{\Omega}\setminus \{y\})
\quad \text{for any }\, y\in \Omega.
 \end{equation}
Let $x,y\in \Omega$ with $x\neq y$,
$0<R\le \min\{1,|x-y|/2\}$, and $\varepsilon_\rho\in (0, R]$.
Since it holds that
$$
\left\{
\begin{aligned}
\cL \cG_{\varepsilon_\rho}(\cdot,y)+\nabla \cP_{\varepsilon_\rho}(\cdot,y)=0 &\quad \text{in }\, \Omega_R(x),\\
\operatorname{div} \cG_{\varepsilon_\rho}(\cdot,y)=-|\Omega|^{-1} & \quad \text{in }\, \Omega_R(x),\\
\cG_{\varepsilon_\rho}(\cdot,y)=0 &\quad \text{on }\, \partial \Omega\cap B_R(x),
\end{aligned}
\right.
$$
by using \eqref{180304@eq1} and Lemma \ref{180324@lem1}, we have
\begin{align*}
&\|\cP_{\varepsilon_\rho}(\cdot,y)\|_{L^\infty(\Omega_{R/2}(x))}\\
&\le C \big(  \|\cG_{\varepsilon_\rho}(\cdot,y)\|_{W^{1,1}(\Omega_R(x))}+\|\cP_{\varepsilon_\rho}(\cdot,y)\|_{L^1(\Omega_R(x))}+1\big)\\
&\le C \big(\|\cG_{\varepsilon_\rho}(\cdot,y)\|_{W^{1,1}(\Omega \setminus B_R(y))}+  \|\cP_{\varepsilon_\rho}(\cdot,y)\|_{L^1(\Omega\setminus B_R(y))}+1\big)\\
&\le C.
\end{align*}
From this inequality together with \eqref{180220@eq4a} and the dominated convergence theorem, we see that $\cP^*(y, \cdot)\in L^1(\Omega_{R/2}(x))$ and
\begin{equation}		\label{180325@eq1}
\lim_{\rho\to \infty}\int_{\Omega_{R/2}(x)} \cP_{\varepsilon_\rho}(z,y)\,dz=\int_{\Omega_{R/2}(x)}\cP^*(y,z)\,dz<\infty.
\end{equation}
Since $\cP^*(y, \cdot)\in L^1(\Omega_{R/2}(x))$ for all $x\in \Omega\setminus \{y\}$ and $0<R\le \min\{1,|x-y|/2\}$,
we get \eqref{180325@eq2}.

We are ready to complete the proof of \eqref{180220@eq3}.
Fix $y\in \Omega$, and let $x\in \Omega\setminus \{y\}$ be a Lebesgue point of $\cP^*(y,\cdot)$.
For $0<R\le \min\{1,|x-y|/2\}$, we see that  (using \eqref{180121@eq2})
$$
\lim_{\rho\to \infty}\int_{\Omega_{R/2}(x)} \cP_{\varepsilon_\rho}(z,y)\,dz=\int_{\Omega_{R/2}(x)}\cP(z,y)\,dz.
$$
Combining this identity and \eqref{180325@eq1}, we get
$$
\dashint_{\Omega_{R/2}(x)}\cP(z,y)\,dz=\dashint_{\Omega_{R/2}(x)}\cP^*(y,z)\,dz.
$$
Therefore, by taking $R\to 0$, and using the continuity of $\cP(\cdot,y)$ on $\Omega\setminus \{y\}$, we concluded that
$$
\cP(x,y)=\cP^*(y,x).
$$
This implies \eqref{180220@eq3}.
The theorem is proved.
\qed

%========================================
\subsection{Proof of Theorem \ref{M5}}
%========================================
We only consider the case when $\Omega=\bR^d_+$ with $d\ge 2$.
Let $(\cG, \cP)$ be the Green function for the pressure of $\cL$ in $\Omega$ constructed in Theorem \ref{M1}.
Then by the same reasoning as in the proof of the estimates in Remark \ref{180218@rmk1} (using Lemma \ref{180227@lem2} instead of Lemma \ref{180208@lem1}), we have the following estimates for any $y\in \Omega$:
\begin{enumerate}[$(a)$]
\item
For any $R\in (0,1] $, we have that
$$
\|\cG(\cdot,y)\|_{L^{q^*}(\Omega\setminus \overline{B_R(y)})}\le CR^{-d/2},
$$
$$
\|D\cG(\cdot,y)\|_{L^{2}(\Omega\setminus \overline{B_R(y)})}+\|\cP(\cdot,y)\|_{L^{2}(\Omega\setminus \overline{B_R(y)})}\le CR^{-d/2},
$$
where $q^*=2d/(d-2)$ if $d\ge 3$ and $q^*=\infty$ if $d=2$.
\item
For any $t>0$, we have that
$$
\big|\{x\in \Omega:|\cG(x,y)> t\}\big|\le Ct^{-d/(d-1)},
$$
$$
\big|\{x\in \Omega:|D_x \cG (x,y)|>t\}\big|+\big|\{x\in \Omega:| \cP (x,y)|>t\}\big|\le Ct^{-1}.
$$
\item
For any $R\in (0,1]$, we have that
$$
\|\cG (\cdot,y)\|_{L^q(\Omega_R(y))}\le C_qR^{1-d+d/q}, \quad 0< q<\frac{d}{d-1},
$$
$$
\|D\cG (\cdot,y)\|_{L^q(\Omega_R(y))}+\|\cP (\cdot,y)\|_{L^q(\Omega_R(y))}\le C_q R^{-d+d/q}, \quad 0<q<1.
$$
\end{enumerate}
In the above, $C=C(d,\lambda, \omega_A)$ and $C_q$ depends also on $q$.
Observe from Definition \ref{D0} $(b)$ that for any $x\in \Omega$ and $R\in (0,1]$ satisfying $|x-y|\ge R$,  $(\cG(\cdot,y), \cP(\cdot,y))$ satisfies
$$
\left\{
\begin{aligned}
\cL \cG(\cdot,y)+\nabla \cP(\cdot,y)=0 &\quad \text{in }\, \Omega_R(x),\\
\operatorname{div}\cG(\cdot,y)=0 &\quad \text{in }\, \Omega_R(x),\\
\cG(\cdot,y)=0 &\quad \text{on }\, \partial \Omega \cap B_R(x).
\end{aligned}
\right.
$$
By Lemma \ref{180227@lem2}, H\"older's inequality, and the estimates in $(a)$, we obtain that
\begin{align}
\nonumber
&R^{-1}\|\cG(\cdot,y)\|_{L^\infty(\Omega_{R/4}(x))}+\|D\cG(\cdot,y)\|_{L^\infty(\Omega_{R/4}(x))}+\|\cP(\cdot,y)\|_{L^\infty(\Omega_{R/4}(x))}\\
\nonumber
&\le CR^{-d/2}\|\cG(\cdot,y)\|_{L^{q^*}(\Omega_{R/4}(x))}\\
\nonumber
&\quad +CR^{-d/2}\big(\|D\cG(\cdot,y)\|_{L^2(\Omega_{R/2}(x))}+\|\cP(\cdot,y)\|_{L^2(\Omega_{R/2}(x))}\big)\\
\nonumber
&\le C R^{-d/2}\|\cG(\cdot,y)\|_{L^{q^*}(\Omega\setminus \overline{B_{R/2}(y)})}\\
\nonumber
&\quad +C R^{-d/2}\big(\|D\cG(\cdot,y)\|_{L^2(\Omega\setminus \overline{B_{R/2}(y)})}+\|\cP(\cdot,y)\|_{L^2(\Omega\setminus \overline{B_{R/2}(y)})}\big)\\
\label{180309@A2}
&\le CR^{-d},
\end{align}
where $C=C(d,\lambda,\omega_A)$.
Since the above inequality holds for any $x\in \Omega$ and $R\in (0,1]$ satisfying $|x-y|\ge R$,  we see that
$$
(\cG(\cdot,y), \cP(\cdot,y))\in W^{1,\infty}(\Omega\setminus \overline{B_r(y)})^d\times L^\infty(\Omega\setminus \overline{B_r(y)}), \quad \forall r>0,
$$
which gives \eqref{180724@A1}.

To verify \eqref{180227@eq1} and \eqref{180302@eq2}, let $x,y\in \Omega$ with $0<|x-y|=R\le 1$.
Then we get \eqref{180302@eq2} from \eqref{180309@A2} immediately.
We also have that
\begin{equation}		\label{180326@eq1}
|G(x,y)|\le C R^{1-d}= C|x-y|^{1-d}.
\end{equation}
If $d_x<R/4$, then we  take $x_0\in \partial \Omega$ such that $\operatorname{dist}(x,x_0)=d_x$.
Since $\cG(x_0,y)=0$, we obtain by \eqref{180309@A2} that
$$
|\cG(x,y)|=|\cG(x,y)-\cG(x_0, y)|\le Cd_x \|D\cG(\cdot,y)\|_{L^\infty (\Omega_{R/4}(x))}\le Cd_x |x-y|^{-d}.
$$
This together with \eqref{180326@eq1} yields that
$$
|\cG(x,y)|\le C\min\{d_x,|x-y|/4\} \cdot |x-y|^{-d},
$$
which gives \eqref{180227@eq1}.
The theorem is proved.
\qed

%========================================
\section{Green function for the flow velocity}	\label{S7}
%========================================
In this section, we deal with Green function and fundamental solution for the flow velocity of Stokes system.
In the definitions below, $G=G(x,y)$ is a $d\times d$ matrix-valued function and $\Pi=\Pi(x,y)$ is a $d\times 1$ vector-valued function on $\Omega\times \Omega$.

\begin{definition}[Green function for the flow velocity]		\label{D2}
Let $\Omega$ be a domain in $\bR^d$.
We say that $(G, \Pi)$ is a Green function for the flow velocity of $\cL$ in $\Omega$ if it satisfies the following properties.
\begin{enumerate}[$(a)$]
\item
For any $y\in \Omega$ and $R>0$,
$$
G(\cdot,y)\in W^{1,1}_{\rm{loc}} (\Omega)^{d\times d}, \quad (1-\eta)G(\cdot,y)\in Y^{1,2}_0(\Omega)^{d\times d},
$$
where $\eta$ is a smooth function satisfying $\eta\equiv 1$ on $B_R(y)$.
Moreover,
$$
\Pi(\cdot,y)\in L^1_{\rm{loc}}(\Omega)^d \cap L^2(\Omega\setminus \overline{B_R(y)})^d.
$$
\item
For any $y\in \Omega$, $(G(\cdot,y), \Pi(\cdot,y))$ satisfies
$$
\left\{
\begin{aligned}
\cL G(\cdot,y)+\nabla \Pi(\cdot,y)=\delta_y I &\quad \text{in }\, \Omega, \\
\operatorname{div} G(\cdot,y)=0 &\quad \text{in }\, \Omega,\\
G(\cdot,y)=0 &\quad \text{on }\, \partial \Omega.
\end{aligned}
\right.
$$
\item
If $(u, p)\in Y^{1,2}_0(\Omega)^d\times \tilde{L}^2(\Omega)$ is a weak solution of the problem
$$
\left\{
\begin{aligned}
\cL^* u+\nabla p=f +D_\alpha f_\alpha&\quad \text{in }\, \Omega,\\
\operatorname{div} u=g-(g)_{\Omega} &\quad \text{in }\, \Omega,\\
u=0 &\quad \text{on }\, \partial \Omega,
\end{aligned}
\right.
$$
where $f, f_\alpha\in C^\infty_0(\Omega)^d$,  and $g\in C^\infty_0(\Omega)$,
then for a.e. $y\in \Omega$, we have
\begin{equation}		\label{180222@A1}
u(y)=\int_\Omega G(x,y)^{\top} f(x)\,dx-\int_\Omega D_\alpha G(x,y)^{\top}f_\alpha(x)\,dx-\int_\Omega \Pi(x,y)g(x)\,dx.
\end{equation}
\end{enumerate}
The Green function for the adjoint operator $\cL^*$ is defined similarly.
The Green function in $\Omega=\bR^d$ is called {\em{the fundamental solution}}.
\end{definition}

\begin{remark}
In the definitions above,
$\cL G(\cdot,y)+\nabla \Pi(\cdot,y)=\delta_y I$ is understood as
$$
\int_\Omega A^{\alpha\beta}_{ij} D_\beta G^{j k}(x,y)D_\alpha \phi^i\,dx+\int_\Omega \Pi^k (x,y)\operatorname{div} \phi\,dx=-\phi^k (y)
$$
for any $\phi\in C^\infty_0(\Omega)^d$ and $k\in \{1,\ldots,d\}$.
\end{remark}

%========================================
\subsection{Main results}		
%========================================

In this subsection, we state the main results concerning Green function $(G, \Pi)$ for the flow velocity.
Note that in \cite{MR3693868}, the authors proved the global pointwise bound
$$
|G(x,y)|\le C|x-y|^{2-d}, \quad x\neq y,
$$
when the coefficients are VMO in a bounded $C^1$ domain.
See also \cite{MR3670039} for the corresponding results in unbounded domains.
In the theorem below, we extend the results in \cite{MR3693868} and \cite{MR3670039} by showing the pointwise bounds \eqref{180328@C1} and  \eqref{180328@C2} under the stronger assumption that the coefficients are of Dini mean oscillation in a domain having a $C^{1,\rm{Dini}}$ boundary.

\begin{theorem}		\label{M3}
Let $d\ge 3$ and $\Omega$ be a domain in $\bR^d$ having a $C^{1,\rm{Dini}}$ boundary as in Definition \ref{D6}.
Suppose that the coefficients $A^{\alpha\beta}$ of $\cL$ are of Dini mean oscillation in $\Omega$ satisfying Definition \ref{D1} $(b)$ with a Dini function $\omega=\omega_A$.
Then under Assumption \ref{A1}, there exists a unique Green function $(G, \Pi)$ for the flow velocity  of $\cL$ in $\Omega$ such that
for any $y\in \Omega$,
\begin{equation}		\label{180730@eq1}
G(\cdot,y)\, \text{ is continuously differentiable in }\, \overline{\Omega}\setminus \{y\}
\end{equation}
and
\begin{equation}		\label{180730@eq2}
\Pi(\cdot,y) \, \text{ is continuous in }\, \overline{\Omega}\setminus \{y\}.
\end{equation}
Moreover, for any $x,y\in \Omega$ with $0<|x-y|\le 1$, we have that
\begin{equation}		\label{180328@C1}
|G(x,y)|\le C \min\{d_x, |x-y|\}\cdot \min\{d_y,|x-y|\} \cdot |x-y|^{-d},
\end{equation}
\begin{equation}		\label{180328@C2}
|D_xG(x,y)|+|\Pi(x,y)|\le C|x-y|^{1-d},
\end{equation}
where $C=C(d,
\lambda,\omega_A, K_0, R_0, \varrho_0)$.
Furthermore, if $(G^*, \Pi^*)$ is the Green function for the flow velocity of $\cL^*$ in $\Omega$, then we have
\begin{equation}		\label{180727@A2}
G(x,y)=G^*(y,x)^{\top} \quad \text{for all }\, x,y\in \Omega, \quad x\neq y.
\end{equation}
\end{theorem}

The corresponding results for the case with $d=2$ was proved in \cite{MR3906316}.

\begin{theorem}{\cite[Theorems 3.2 and 3.7]{MR3906316}}		\label{180729@thm1}
Let  $\Omega$ be a bounded domain in $\bR^2$ having a $C^{1,\rm{Dini}}$ boundary as in Definition \ref{D6}.
Suppose that the coefficients $A^{\alpha\beta}$ of $\cL$ are of Dini mean oscillation in $\Omega$ satisfying Definition \ref{D1} $(b)$ with a Dini function $\omega=\omega_A$.
Then there exists a unique Green function $(G, \Pi)$ for the flow velocity  of $\cL$ in $\Omega$ such that
for any $y\in \Omega$,
$$
G(\cdot,y)\, \text{ is continuously differentiable in }\, \overline{\Omega}\setminus \{y\}
$$
and
$$
\Pi(\cdot,y) \, \text{ is continuous in }\, \overline{\Omega}\setminus \{y\}.
$$
Moreover, for any $x,y\in \Omega$ with $x\neq y$, we have
$$
|G(x,y)|\le C\bigg(1+\log\bigg(\frac{\operatorname{diam}(\Omega)}{|x-y|}\bigg)\bigg),
$$
$$
|D_xG(x,y)|+|\Pi(x,y)|\le C|x-y|^{-1},
$$
where $C=C(
\lambda,\omega_A, R_0, \varrho_0, \operatorname{diam}(\Omega))$.
Furthermore, if $(G^*, \Pi^*)$ is the Green function for the flow velocity of $\cL^*$ in $\Omega$, then we have
\begin{equation}		\label{180729@eq2}
G(x,y)=G^*(y,x)^{\top} \quad \text{for all }\, x,y\in \Omega, \quad x\neq y.
\end{equation}
\end{theorem}

Based on Theorems \ref{M2}, \ref{M3}, and \ref{180729@thm1}, we have the following corollary, the proof of which is given  in Section \ref{S7-3}.

\begin{corollary}		\label{180727@cor1}
Suppose that  the same hypothesis of Theorem \ref{M3} (resp. Theorem \ref{180729@thm1}) hold.
Let $(G, \Pi)$ and $(\cG, \cP)$ be the Green functions for the flow velocity and the pressure of $\cL$ in $\Omega$ derived from Theorems \ref{M3} (resp. Theorem \ref{180729@thm1}) and \ref{M1}, respectively.
Then for $f\in C^\infty_0(\Omega)^d$, the pair $(u, p)$ given by
\begin{equation}		\label{180729@eq5}
u(y)=\int_\Omega G(x,y)^{\top} f(x)\,dx, \quad p(y)=-\int_\Omega \cG(x,y)\cdot f(x)\,dx
\end{equation}
is a unique weak solution in $Y^{1,2}_0(\Omega)^d\times \tilde{L}^2(\Omega)$ of the problem
$$
\left\{
\begin{aligned}
\cL^* u+\nabla p=f &\quad \text{in }\, \Omega,\\
\operatorname{div} u=0 &\quad \text{in }\, \Omega,\\
u=0 &\quad \text{on }\, \partial \Omega.
\end{aligned}
\right.
$$
Moreover, if we define $(d+1)\times (d+1)$ matrix-valued functions by
$$
\mathbf{G}=
\begin{pmatrix}
G^{11} & G^{12} & \ldots & G^{1d} &-\cG^{1}\\
\vdots & \vdots & \ddots & \vdots & \vdots\\
G^{d1} & G^{d2} & \ldots & G^{dd} & -\cG^{d}\\
\Pi^1 & \Pi^2 & \ldots & \Pi^d & \cP
\end{pmatrix}
$$
and
$$
\mathbf{G^*}=
\begin{pmatrix}
(G^*)^{11} & (G^*)^{12} & \ldots & (G^*)^{1d} &-(\cG^*)^{1}\\
%G^{21} & G^{22} & \ldots & G^{2d} &-\cG^{2}\\
\vdots & \vdots & \ddots & \vdots & \vdots\\
(G^*)^{d1} & (G^*)^{d2} & \ldots & (G^*)^{dd} & -(\cG^*)^{d}\\
(\Pi^*)^1 & (\Pi^*)^2 & \ldots & (\Pi^*)^d & \cP^*
\end{pmatrix},
$$
where $(G^*, \Pi^*)$ and $(\cG^*, \cP^*)$ are the Green functions for the flow velocity and the pressure of $\cL^*$ in $\Omega$, respectively, then for any $y\in \Omega$, there exists a measure zero set $N_y\subset\Omega$ containing $y$ such that
we have
\begin{equation}		\label{180727@A1}
\mathbf{G}(x,y)=\mathbf{G}^*(y,x)^{\top} \quad \text{for all }\, x\in \Omega\setminus N_y.
\end{equation}
\end{corollary}

%========================================
\subsection{Proof of Theorem \ref{M3}}		\label{S7-2}
%========================================

We first prove the existence of the Green function for the flow velocity.
In the case when $|\Omega|<\infty$, we shall follow the arguments in \cite{MR3693868}, where the authors proved the existence of the Green function in a bounded Lipschitz domain under the following assumption.
\begin{assumption}[\text{$(\bf{A1})$ in \cite[Section 2.1]{MR3693868}}]		\label{180328@A2}
There exists constants $\mu\in (0,1]$ and $A_1>0$ such that the following holds:
If $(u, p)\in W^{1,2}(B_R(x_0))^d\times L^2(B_R(x_0))$ satisfies either
$$
\left\{
\begin{aligned}
\cL u+\nabla p=0 \quad \text{in }\, B_R(x_0),\\
\operatorname{div} u=0 \quad \text{in }\, B_R(x_0),
\end{aligned}
\right.
\qquad \text{or}\qquad
\left\{
\begin{aligned}
\cL^* u+\nabla p=0 \quad \text{in }\, B_R(x_0),\\
\operatorname{div} u=0 \quad \text{in }\, B_R(x_0),
\end{aligned}
\right.
$$
where $x_0\in \Omega$ and $R\in (0, d_{x_0}]$, then we have
\begin{equation}		\label{180328@A3}
\sup_{x,y\in B_{R/2}(x_0)}\frac{|u(x)-u(y)|}{|x-y|^\mu}\le A_1 R^{-\mu}\bigg(\dashint_{B_R(x_0)}|u|^2\,dx\bigg)^{1/2}.
\end{equation}
\end{assumption}

Note that because of $|\Omega|<\infty$, we have
$$
\sup_{x\in \Omega}d_x \le M(|\Omega|)<\infty.
$$
Hence, under the hypothesis of Theorem \ref{M3},
we can show that Assumption \ref{180328@A2} holds with $\mu=1$ and $A_1=A_1(d,\lambda, \omega_A, M)$.
Indeed, since $(u, p-(p)_{B_{3R/4}(x_0)})$ satisfies the same system, by \eqref{180110@eq8} with a covering argument and H\"older's inequality, we have
\begin{align*}
&\sup_{x,y\in B_{R/2}(x_0)}\frac{|u(x)-u(y)|}{|x-y|} \le \|Du\|_{L^\infty(B_{R/2}(x_0))}\\
&\le CR^{-d/2}\big(\|Du\|_{L^2(B_{3R/4}(x_0))}+\|p-(p)_{B_{3R/4}(x_0)}\|_{L^2(B_{3R/4}(x_0))}\big),
\end{align*}
where $C=C(d,\lambda, \omega_A, M)$.
Thus we get \eqref{180328@A3} from the above inequality and  Caccioppoli's inequality (see, for instance, \cite[Lemma 3.3]{MR3693868}).
Moreover, it is easy to check that, under Assumptions \ref{A1} and \ref{180328@A2}, the proof of \cite[Theorem 2.3]{MR3693868} still works for the domain $\Omega$.
Therefore, by the existence result in \cite[Theorem 2.3]{MR3693868} of a Green function,
there exist Green functions $(G, \Pi)$ and $(G^*, \Pi^*)$ for the flow velocity of $\cL$ and $\cL^*$, respectively, satisfying the properties in Definition \ref{D2} and \eqref{180727@A2}.
Notice from Definition \ref{D2} $(b)$ that
\begin{equation}		\label{180329_eq1}
\left\{
\begin{aligned}
\cL G(\cdot,y)+\nabla \Pi(\cdot,y)=0 &\quad \text{in }\, \Omega_R(x),\\
\operatorname{div} G(\cdot,y)=0 &\quad \text{in }\, \Omega_R(x),\\
G(\cdot,y)=0 &\quad \text{on }\, \partial \Omega \cap B_R(x)
\end{aligned}
\right.
\end{equation}
for any $x\in \Omega$ and $R>0$ satisfying $y\notin B_R(x)$.
Thus by \eqref{180213@eq1}, we get \eqref{180730@eq1} and \eqref{180730@eq2}.
Similarly, the existence of the Green function in a domain $\Omega$ with $|\Omega|=\infty$ follows from
\cite[Theorem 10.4]{MR3670039}.

We now turn to the proof of  \eqref{180328@C2}.
Let $x,y\in \Omega$ with $0<|x-y|\le 1$, and set $R=|x-y|/2$.
Suppose that  $(u, p)\in Y^{1,2}_0(\Omega)^d\times \tilde{L}^2(\Omega)$ is the weak solution of
$$
\left\{
\begin{aligned}
\cL^* u+\nabla p=f&\quad \text{in }\, \Omega,\\
\operatorname{div} u=0 &\quad \text{in }\, \Omega,
\end{aligned}
\right.
$$
where $f\in C^\infty_0(\Omega\setminus \overline{B_R(y)})^d$.
Since $f\equiv 0$ in $\Omega_R(y)$, by \eqref{180304@eq1} and
Lemma \ref{180108@lem1},  we have
\begin{align*}
\|u\|_{L^\infty (\Omega_{R/2}(y))}&\le CR^{-d}\|u\|_{L^1(\Omega_R(y))}+CR^{1-d}\big(\|Du\|_{L^1(\Omega_R(y))}+\|p\|_{L^1(\Omega_R(y))}\big)\\
&\le CR^{1-d/2}\big(\|u\|_{L^{2d/(d-2)}(\Omega)}+\|Du\|_{L^2(\Omega)}+\|p\|_{L^2(\Omega)}\big)\\
&\le CR^{1-d/2} \|f\|_{L^{2d/(d+2)}(\Omega\setminus \overline{B_R(y)})},
\end{align*}
where $C=C(d,\lambda, \omega_A,  K_0, R_0, \varrho_0)$.
Since $u$ is continuous at $y$, we get from \eqref{180222@A1} and the above inequality that
$$
\bigg|\int_{\Omega\setminus \overline{B_R(y)}} G(x,y)^\top f(x)\,dx\bigg|\le CR^{1-d/2}\|f\|_{L^{2d/(d+2)}(\Omega\setminus \overline{B_R(y)})}.
$$
Thus by the duality, we have
\begin{equation}		\label{180222@C2}
\|G(\cdot,y)\|_{L^{2d/(d-2)}(\Omega\setminus \overline{B_R(y)})}\le CR^{1-d/2}.
\end{equation}
Similarly, we obtain
\begin{equation}		\label{180222@C3}
\|DG(\cdot,y)\|_{L^{2}(\Omega\setminus \overline{B_R(y)})}+\|\Pi(\cdot,y)\|_{L^{2}(\Omega\setminus \overline{B_R(y)})}\le CR^{1-d/2}.
\end{equation}
From \eqref{180222@C2}, \eqref{180222@C3}, and  \eqref{180304@eq1} applied to \eqref{180329_eq1}, we get
\begin{equation}		\label{180328@C3}
\begin{aligned}
&R^{-1}\|G(\cdot,y)\|_{L^\infty(\Omega_{R/2}(x))}+\|DG(\cdot,y)\|_{L^\infty(\Omega_{R/2}(x))}\\
&\quad +\|\Pi(\cdot,y)\|_{L^\infty(\Omega_{R/2}(x))}\le CR^{1-d}.
\end{aligned}
\end{equation}
This gives \eqref{180328@C2} and
\begin{equation}		\label{180328@C4}
|G(x,y)|\le C|x-y|^{2-d}.
\end{equation}

To prove \eqref{180328@C1}, we use the idea in the proof of \cite[Theorem 3.13]{MR2718661}, where the authors obtained pointwise bounds for Green functions of elliptic systems near the boundary.
We claim that for any $x,y\in \Omega$ with $0<|x-y|\le 1$, we have
\begin{equation}		\label{180329@A1}
|G(x,y)|\le C\min\{d_x,|x-y|\}\cdot |x-y|^{1-d},
\end{equation}
\begin{equation}		\label{180329@A2}
|G(x,y)|\le C\min\{d_y,|x-y|\}\cdot |x-y|^{1-d}.
\end{equation}
We denote $R=|x-y|/2$ and extend $G(\cdot,y)$ by zero on $\bR^d\setminus \Omega$.
Assume $d_x<R/2$, and take $x_0\in \partial \Omega$ such that $\operatorname{dist}(x,x_0)=d_x$.
Since $G(x_0,y)=0$, we obtain by \eqref{180328@C3} that
\begin{equation}		\label{180329@A5}
|G(x,y)|=|G(x,y)-G(x_0,y)|\le d_x \|DG(\cdot,y)\|_{L^\infty(\Omega_{R/2}(x))} \le Cd_x |x-y|^{1-d}.
\end{equation}
From \eqref{180328@C4} and \eqref{180329@A5}, we get
$$
|G(x,y)|\le  Cd_x R^{1-d}\le C\min\{d_x,|x-y|/4\}\cdot |x-y|^{1-d},
$$
which gives \eqref{180329@A1}.
By the same reasoning, we have
$$
|G^*(y,x)|\le C\min\{d_y, |x-y|\}\cdot |x-y|^{1-d}.
$$
This together with \eqref{180727@A2} yields \eqref{180329@A2}.

We are ready to prove \eqref{180328@C1}.
We again let $x,y\in \Omega$ with $x\neq y$, and set $R=|x-y|/2$.
Assume $d_x<R/16$, and we take $x_0\in \partial \Omega$ such that $\operatorname{dist}(x,x_0)=d_x$.
Note that
$$
\Omega_{R/8}(x)\subset \Omega_{R/4}(x_0)\subset \Omega_{R/2}(x_0)\subset \Omega_R(x).
$$
Since $(G(\cdot,y), \Pi(\cdot,y))$ satisfies \eqref{180329_eq1},  we have the following boundary Caccioppoli inequality
$$
\|DG(\cdot,y)\|_{L^2(\Omega_{R/4}(x_0))}+\|\Pi(\cdot,y)\|_{L^2(\Omega_{R/4}(x_0))} \le CR^{-1}\|G(\cdot,y)\|_{L^2(\Omega_{R/2}(x_0))}.
$$
Using this together with \eqref{180304@eq1} and H\"older's inequality, we have
\begin{align}
\nonumber
&|G(x,y)|=|G(x,y)-G(x_0,y)|\le d_x \|DG(\cdot,y)\|_{L^\infty(\Omega_{R/16}(x))}\\
\nonumber
&\le Cd_x R^{-d/2-1}\|G(\cdot,y)\|_{L^2(\Omega_{R/4}(x_0))}\\
\nonumber
&\quad +Cd_xR^{-d/2}\big(\|DG(\cdot,y)\|_{L^2(\Omega_{R/4}(x_0))}+\|\Pi(\cdot,y)\|_{L^2(\Omega_{R/4}(x_0))}\big)\\
\nonumber
&\le Cd_xR^{-d/2-1}\|G(\cdot,y)\|_{L^2(\Omega_{R/2}(x_0))}\\
\label{180329@A6}
&\le Cd_xR^{-1}\|G(\cdot,y)\|_{L^\infty(\Omega_{R}(x))}.
\end{align}
Note that for any $z\in \Omega_{R}(x)$, we have $R<|z-y|<3R$.
Thus, it follows from \eqref{180329@A2} that
$$
\begin{aligned}
|G(z,y)| &\le C\min\{d_y,|z-y|\}\cdot |z-y|^{1-d}\\
&\le C\min\{d_y,|x-y|\}\cdot |x-y|^{1-d}.
\end{aligned}
$$
This together with \eqref{180329@A6} yields
$$
|G(x,y)|\le Cd_x \min\{d_y,|x-y|\}\cdot |x-y|^{-d}.
$$
Finally, combining \eqref{180329@A2} and the above inequality, we get the desired estimate \eqref{180328@C1}.
The theorem is proved.
\qed

%========================================
\subsection{Proof of Corollary \ref{180727@cor1}}			\label{S7-3}
%========================================
The representation formula \eqref{180729@eq5} follows immediately from Definition \ref{D0} $(c)$ and  Definition \ref{D2} $(c)$.
To verify \eqref{180727@A1},
let $x,y\in \Omega$ with $x\neq y$ and  $(\cG_\sigma^*(\cdot,x), \cP_\sigma^*(\cdot,x))\in Y^{1,2}_0(\Omega)^d\times \tilde{L}^2(\Omega)$ be the approximated Green function for the pressure of $\cL^*$ satisfying \eqref{180222@E1}.
Then by Definition \ref{D2} $(c)$ and the continuity of $\cG_\sigma^*(\cdot,x)$, we have
\begin{equation}		\label{180223@eq2}
\cG^*_\sigma(y,x)=-\int_{\Omega} \Pi(z,y) \Phi_{\sigma,x}(z)\,dz.
\end{equation}
Due to the continuity of $\Pi(\cdot,y)$, the right-hand side of \eqref{180223@eq2} converges to $-\Pi(x,y)$ as $\sigma \to 0$.
On the other hand, by the counterpart of Lemma \ref{180223@lem1}, there exists a subsequence of $\{\cG_\sigma^*(y,x)\}$ that converges to $\cG^*(y,x)$.
Therefore, we conclude that
$$
\cG^*(y,x)=-\Pi(x,y).
$$
From this together with \eqref{180220@eq3} and \eqref{180727@A2} (resp. \eqref{180729@eq2}), we conclude that \eqref{180727@A1} holds when $d\ge 3$ (resp. $d=2$).
\qed

\appendix

%========================================
\section{$L^\infty$-estimates}	\label{S8}
%========================================
In this section, we prove $L^\infty$-estimates of solutions and its derivatives, which are crucial for proving our main results. Denote $B_r=B_r(0)$ for any $r>0$.
The following lemma concerns interior estimates, the proof of which is based on the $W^{1,\infty}$-regularity result in \cite{MR3912724}.

\begin{lemma}		\label{180108@lem3}
Let $R\in (0,1]$.
Suppose that the coefficients $A^{\alpha\beta}$ of $\cL$ are of partially Dini mean oscillation with respect to $x'$ in the interior of $B_R$ satisfying Definition \ref{D1} $(a)$ with a Dini function $\omega=\omega_A$.
If $(u,p)\in W^{1,2}(B_R)^d\times L^2(B_R)$ satisfies
$$
\left\{
\begin{aligned}
\cL u+\nabla p=f &\quad \text{in }\, B_R,\\
\operatorname{div} u=\ell &\quad \text{in }\, B_R,
\end{aligned}
\right.
$$
where $f\in L^\infty(B_R)^d$ and $\ell\in \bR$, then we have
\begin{equation}		\label{180406@eq2}
(u,p)\in W^{1,\infty}(B_{R/2})^d\times L^\infty(B_{R/2})
\end{equation}
with the estimates
\begin{equation}		\label{180110@eq8}
\begin{aligned}
&\|Du\|_{L^\infty(B_{R/2})}+\|p\|_{L^\infty(B_{R/2})}\\
&\le C\big( R^{-d}\big(\|Du\|_{L^1(B_R)}+\|p\|_{L^1(B_R)}\big)+R\|f\|_{L^\infty(B_R)}+|\ell|\big),
\end{aligned}
\end{equation}
and
\begin{equation}		\label{180110@eq8a}
\begin{aligned}
&\|u\|_{L^\infty(B_{R/2})} \le CR^{-d}\|u\|_{L^1(B_{R/2})}\\
&\quad +C\big(R^{1-d}\big(\|Du\|_{L^1(B_R)}+\|p\|_{L^1(B_R)}\big)+R^2\|f\|_{L^\infty(B_R)}+R|\ell|\big),
\end{aligned}
\end{equation}
where $C=C(d,\lambda, \omega_A)$.
If we further assume that $A^{\alpha\beta}$ are of Dini mean oscillation with respect to all direction in $B_R$ satisfying Definition \ref{D1} $(b)$, then we have
\begin{equation}		\label{180406@eq1}
(u, p)\in C^1(\overline{B_{R/2}})^d\times C(\overline{B_{R/2}}).
\end{equation}
\end{lemma}

\begin{proof}
Note that \eqref{180406@eq2} and \eqref{180406@eq1} are easy consequences of \cite[Theorems 2.2 and 2.3, and Remark 2.4]{MR3912724} together with covering and scaling arguments.
Moreover, the estimate \eqref{180110@eq8a} follows from \eqref{180110@eq8}.
Indeed, using the following Poincar\'e inequality
\begin{equation}		\label{180213@eq2a}
\|u-(u)_{B_{R/2}}\|_{L^\infty(B_{R/2})}\le C R\|Du\|_{L^\infty(B_{R/2})},
\end{equation}
we have
$$
\|u\|_{L^\infty(B_{R/2})}\le C\big( R^{-d}\|u\|_{L^1(B_{R/2})}+R\|Du\|_{L^\infty(B_{R/2})}\big),
$$
where $C=C(d)$.
This inequality together with \eqref{180110@eq8} implies \eqref{180110@eq8a}.
Thus, to complete the proof of the lemma, we only need to prove \eqref{180110@eq8}.
By a covering argument, it suffices to show that
\begin{equation}		\label{180108@eq1}
\begin{aligned}
&\|Du\|_{L^\infty(B_{R/3})}+\|p\|_{L^\infty(B_{R/3})}\\
&\le C\big(R^{-d}\big(\|Du\|_{L^1(B_R)}+\|p\|_{L^1(B_R)}\big)+R\|f\|_{L^\infty(B_R)}+|\ell|\big),
\end{aligned}
\end{equation}
where $C=C(d,\lambda, \omega_A)$.
Let us fix $q>d$.
By \cite[Lemma 3.1]{MR3809039}, there exist functions
$f_\alpha\in \tilde{W}^{1,q}(B_R)^d$, $\alpha\in\{1,\ldots,d\}$, such that
$$
D_\alpha f_\alpha=f \quad \text{in }\, B_R, \quad \|Df_\alpha\|_{L^q(B_R)}\le C(d,q)\|f\|_{L^q(B_R)},
$$
which together with the Morrey inequality implies
\begin{equation}		\label{180108@eq2a}
\|f_\alpha\|_{L^\infty(B_R)}+R^{1-d/q}[f_\alpha]_{C^{1-d/q}(B_R)}\le C(d,q)  R\|f\|_{L^\infty(B_R)}.
\end{equation}
Note that $(u, p)$ satisfies
\begin{equation}		\label{180228@eq2}
\left\{
\begin{aligned}
\cL u+\nabla p=D_\alpha f_\alpha &\quad \text{in }\, B_R,\\
\operatorname{div} u=\ell &\quad \text{in }\, B_R.
\end{aligned}
\right.
\end{equation}
We now apply \cite[Theorem 2.2 $(a)$]{MR3912724} with the $L^\infty$-estimate \cite[Eq. (4.16)]{MR3912724} to the scaled system of \eqref{180228@eq2}.
To this end, let
$$
\hat{u}(y)=u(x), \quad \hat{p}(y)=p(x), \quad \hat{A}^{\alpha\beta}(y)=A^{\alpha\beta}(x), \quad \hat{f}_\alpha(y)=f_\alpha (x),
$$
where $y=6x/R$, and observe that
\begin{equation}		\label{180109@B1}
\left\{
\begin{aligned}
D_\alpha(\hat{A}^{\alpha\beta}D_\beta \hat{u})+\nabla (R\hat{p}/6)=D_\alpha (R\hat{f}_\alpha/6) &\quad \text{in }\, B_6,\\
\operatorname{div} \hat{u}=R\ell/6 &\quad \text{in }\, B_6.
\end{aligned}
\right.
\end{equation}
Fix a constant $\gamma$ satisfying $1-d/q<\gamma<1$, and let $\kappa=\kappa(d,\lambda,\gamma)\in (0,1/2]$ be the constant from \cite[Lemma 4.1]{MR3912724}.
For $r\in (0,1]$ and $h\in L^1(B_6)$, we denote
$$
\begin{aligned}
&\omega_{h,y'}(r)=\sup_{y_0\in B_4} \dashint_{B_r(y_0)}\bigg|h(y)-\dashint_{B_r'(y_0')}h(y_1,z')\,dz'\bigg|\,dy, \\
&\tilde{\omega}_{h,y'}(r)=\sum_{i=1}^\infty \kappa^{\gamma i}\big(\omega_{h,y'}(\kappa^{-i}r)[\kappa^{-i}r<1]+\omega_{h,y'}(1)[\kappa^{-i}r\ge 1]\big),
\end{aligned}
$$
where we used Iverson bracket notation, that is, $[P]=1$ if $P$ is true and $[P]=0$ otherwise.
By \eqref{180108@eq2a} and the fact that $1-d/q<\gamma<1$, we have
$$
R\|\hat{f}_1\|_{L^\infty(B_6)} + \int_0^1 \frac{\tilde{\omega}_{R\hat{f}_\alpha,x'}(t)}{t}\,dt\le C R^2\|f\|_{L^\infty(B_R)},
$$
where $C=C(d,q,\gamma)$.
Using this inequality and \cite[Eq. (4.16) and Remark 4.2]{MR3912724} applied to \eqref{180109@B1}, we find that
$$
\begin{aligned}
&\|D\hat{u}\|_{L^\infty(B_2)}+R\|\hat{p}\|_{L^\infty(B_2)} \\
&\le C\bigg( \|D\hat{u}\|_{L^1(B_6)}+R\|\hat{p}\|_{L^1(B_6)}+R\|\hat{f}_1\|_{L^\infty(B_6)}+R|\ell|+\int_0^1 \frac{\tilde{\omega}_{R\hat{f}_\alpha, x'}(t)}{t}\,dt\bigg),\\
&\le C\big(\|D\hat{u}\|_{L^1(B_6)}+R\|\hat{p}\|_{L^1(B_6)}+R^2\|f\|_{L^\infty(B_R)}+R|\ell|\big),
\end{aligned}
$$
where $C=C(d,\lambda,\omega_0,q,\gamma)$.
Here, $\omega_0:(0,1]\to [0,\infty)$ is a function such that
$$
\int_0^r \frac{\tilde{\omega}_{\hat{A}^{\alpha\beta},y'}(t)}{t}\,dt\le \int_0^r \frac{\omega_0(t)}{t}\,dt <\infty \quad \text{for all }\, r\in (0,1].
$$
Therefore, from the change of variables, we get
$$
\begin{aligned}
&\|Du\|_{L^\infty(B_{R/3})}+\| p \|_{L^\infty(B_{R/3})} \\
&\le C\big( R^{-d}\big(\|Du\|_{L^1(B_R)}+\|p\|_{L^1(B_R)}\big)+R\|f\|_{L^\infty(B_R)}+|\ell|\big),
\end{aligned}
$$
where $C=C(d,\lambda,\omega_0,q,\gamma)=C(d,\lambda, \omega_0)$.

To complete the proof of \eqref{180108@eq1}, it remains to show that $\omega_0$ can be derived from $\omega_A$.
Set
$$
\tilde{\omega}_A(r)=\sum_{i=1}^\infty \kappa^{\gamma i}\big(\omega_A(\kappa^{-i}r)[\kappa^{-i}r<1]+\omega_A(1)[\kappa^{-i}r\ge 1]\big),
$$
and observe that $\tilde{\omega}_A:[0,1]\to [0,\infty)$ is a Dini function; see \cite[Lemma 1]{MR2927619}.
Using the fact that
$$
{\omega}_{\hat{A}^{\alpha\beta},y'}(r)\le \omega_A (rR/6) \quad \text{for all }\, r\in (0,1],
$$
we have
\begin{align}
\nonumber
\tilde{\omega}_{\hat{A}^{\alpha\beta},y'}(r)
&\le \sum_{i=1}^\infty \kappa^{\gamma i} \omega_A(\kappa^{-i}rR/6)[\kappa^{-i}r<1]+\sum_{i=1}^\infty \kappa^{\gamma i} \omega_{A}(R/6)[\kappa^{-i}r\ge 1]\\
\label{180321@eq2}
&=:I_1(r)+I_2(r).
\end{align}
Obviously, $I_1(r)\le \tilde{\omega}_A (rR/6)$.
Since $\omega_A$ is a Dini function, we obtain that (using $R\le 1$)
$$
\omega_A(R/6)\le C \inf_{R/6\le t\le R/3}\omega_A(t)\le C\int_{R/6}^{R/3}\frac{\omega_A(t)}{t}\,dt\le C\int_0^1 \frac{\omega_A(t)}{t}\,dt,
$$
where $C=C(\omega_A)$.
This implies that
$$
I_2(r)\le C \sum_{\kappa^{-i}r\ge 1} \kappa^{\gamma i}\cdot \int_0^1 \frac{\omega_A(t)}{t}\,dt\le C_0 r^{\gamma} \int_0^1 \frac{\omega_A(t)}{t}\,dt,
$$
where $C_0=C_0(\omega_A, \kappa,\gamma)=C_0(d,\lambda, \omega_A)$.
From \eqref{180321@eq2} and the estimates of $I_1$ and $I_2$,  we get
\begin{equation}		\label{180228@eq6}
\tilde{\omega}_{\hat{A}^{\alpha\beta},y'}(r)\le \tilde{\omega}_A(rR/6)+C_0 r^{\gamma} \int_0^1 \frac{\omega_A(t)}{t}\,dt.
\end{equation}
Therefore, for any $r\in (0,1]$, we have
\begin{align*}
\int_0^r \frac{\tilde{\omega}_{\hat{A}^{\alpha\beta},y'}(t)}{t}\,dt
&\le \int_0^r \frac{\tilde{\omega}_A(tR/6)}{t}\,dt+C_0\int_0^r \frac{t^{\gamma}}{t}\,dt\cdot \int_0^1 \frac{\omega_A(t)}{t}\,dt\\
&\le \int_0^r \frac{\omega_0(t)}{t}\,dt<\infty,
\end{align*}
where we set
$$
\omega_0(r)=\tilde{\omega}_A(r)+C_0 r^{\gamma}\int_0^1 \frac{\omega_A(t)}{t}\,dt.
$$
This completes the proof of \eqref{180108@eq1}.
The lemma is proved.
\end{proof}

Based on the $C^1$-regularity result in \cite{MR3890946}, we obtain the following $L^\infty$-estimate  in a domain having a $C^{1,\rm{Dini}}$ boundary.

\begin{lemma}		\label{180208@lem1}
Let $\Omega$ be a (possibly unbounded) domain in $\bR^d$ having a $C^{1,\rm{Dini}}$ boundary as in Definition \ref{D6}.
Suppose that the coefficients $A^{\alpha\beta}$ of $\cL$ are of Dini mean oscillation in $\Omega$ satisfying Definition \ref{D1} $(b)$ with a Dini function $\omega=\omega_A$.
Let $x_0\in \Omega$ and $R\in (0,1]$.
If $(u, p)\in W^{1,2}(\Omega_R(x_0))^d\times L^2(\Omega_R(x_0))$ satisfies
$$
\left\{
\begin{aligned}
\cL u+\nabla p=f &\quad \text{in }\, \Omega_R(x_0),\\
\operatorname{div} u=\ell &\quad \text{in }\, \Omega_R(x_0),\\
u=0 &\quad \text{on }\, \partial \Omega \cap B_R(x_0),
\end{aligned}
\right.
$$
where $f\in L^\infty(\Omega_R(x_0))^d$ and $\ell\in \bR$, then we have
\begin{equation}		\label{180213@eq1}
(u,p)\in C^1(\overline{\Omega_{R/2}(x_0)})^d\times C(\overline{\Omega_{R/2}(x_0)})
\end{equation}
with the estimate
\begin{equation}		\label{180304@eq1}
\begin{aligned}
&R^{-1}\|u\|_{L^\infty(\Omega_{R/2}(x_0))}+\|Du\|_{L^\infty(\Omega_{R/2}(x_0))}+\|p\|_{L^\infty(\Omega_{R/2}(x_0))}\\
&\le CR^{-d-1}\|u\|_{L^1(\Omega_R(x_0))}+ CR^{-d}\big(\|Du\|_{L^1(\Omega_R(x_0))}+\|p\|_{L^1(\Omega_R(x_0))}\big)\\
&\quad +CR\|f\|_{L^\infty(\Omega_R(x_0))}+C|\ell|,
\end{aligned}
\end{equation}
where $C=C(d,\lambda,\omega_A,R_0,\varrho_0)$.
\end{lemma}

\begin{remark}		\label{180321@rmk1}
In the proof of Lemma \ref{180208@lem1} below, we will use the $W^{1,q}$-regularity result in \cite[Corollary 5.3]{MR3693868} (see also \cite{MR3758532}) for the Stokes system.
If $\Omega$ is bounded, then under the hypothesis of Lemma \ref{180208@lem1}, the regularity result is available.
For further details, see the proof of \cite[Theorem 1.6]{MR3890946}.
\end{remark}

\begin{proof}[Proof of Lemma \ref{180208@lem1}]
If $\Omega$ is unbounded, then one may construct a bounded domain $\Omega^*$ such that $\Omega_R(x_0)\subset \Omega^*\subset \Omega$ and it has the same nice properties as $\Omega$.
Thus we may assume that $\Omega$ is bounded.

Let $x_0\in \Omega$ and $R\in (0,1]$.
We first prove \eqref{180213@eq1}.
By using localization and bootstrap arguments combined with the $W^{1,q}$-regularity result in \cite[Corollary 5.3]{MR3693868}, we see that
\begin{equation}		\label{180214@eq1a}
(u, p)\in W^{1,q}(\Omega_r(x_0))^d\times L^q(\Omega_r(x_0)) \quad \text{ for all }\, r\in (0,R) \, \text{ and }\, q\in (1,\infty).
\end{equation}
Let $\eta$ be a smooth function on $\bR^d$ with a compact support in $B_R(x_0)$.
Then, the pair
\begin{equation}		\label{181008@eq5}
(v,\pi)=(\eta u, \eta p-(\eta p)_\Omega)\in W^{1,2}_0(\Omega)^d\times \tilde{L}^2(\Omega)
\end{equation}
satisfies
\begin{equation}		\label{180214@eq1}
\left\{
\begin{aligned}
\cL v+\nabla \pi=h+D_\alpha h_\alpha &\quad \text{in }\, \Omega,\\
\operatorname{div} v= g &\quad \text{in }\, \Omega,\\
v=0 &\quad \text{on }\, \partial \Omega,
\end{aligned}
\right.
\end{equation}
where we set
$$
h=A^{\alpha\beta} D_\beta u D_\alpha \eta +p \nabla \eta+\eta f, \quad h_\alpha=A^{\alpha\beta} D_\beta \eta u, \quad g=\nabla \eta \cdot u+\ell \eta.
$$
Obviously, \eqref{180214@eq1a} implies that $h\in L^q(\Omega)^d$ with $q>d$.
Moreover,
since $u$ is H\"older continuous in $\Omega\cap \operatorname{supp}\eta$, it follows from \cite[Lemma 2.1]{MR3747493} that $h_\alpha$ and $g$ are of Dini mean oscillation in $\Omega$.
Therefore, by \cite[Theorem 1.4 and Remark 1.5]{MR3890946}, we conclude that
$$
(v,\pi)\in C^1(\overline{\Omega})^d\times C(\overline{\Omega}),
$$
which gives \eqref{180213@eq1} if we choose the function $\eta$ such that $\eta\equiv 1$ on $B_{R/2}(x_0)$.

We now turn to the proof of \eqref{180304@eq1}.
By the Poincar\'e inequality \eqref{180213@eq2a} applied to the zero extension of $u$, we have
$$
\|u\|_{L^\infty(\Omega_{R/2}(x_0))}\le CR^{-d}\|u\|_{L^1(\Omega_{R/2}(x_0))}+CR\|Du\|_{L^\infty(\Omega_{R/2}(x_0))}.
$$
Thus, it suffices to prove that
\begin{equation}		\label{180304@eq1a}
\begin{aligned}
&\|Du\|_{L^\infty(\Omega_{R/2}(x_0))}+\|p\|_{L^\infty(\Omega_{R/2}(x_0))}\le CR^{-d-1}\|u\|_{L^1(\Omega_R(x_0))}\\
&\quad + C\big(R^{-d}\big(\|Du\|_{L^1(\Omega_R(x_0))}+\|p\|_{L^1(\Omega_R(x_0))}\big)+R\|f\|_{L^\infty(\Omega_R(x_0))}+|\ell|\big).
\end{aligned}
\end{equation}
Let $y\in \Omega_R(x_0)$ and $r\in (0, R]$ such that $r\le \min\{R_1, \operatorname{dist}(y, \partial B_R(x_0)\}$, where
$R_1=R_1(R_0,\varrho_0)\in (0,R_0/4)$ is the constant from \cite[Lemma 2.2]{MR3890946}.
We use the abbreviations
$$
B_r=B_r(y) \quad \text{and}\quad \Omega_r=\Omega_r(y).
$$
We fix $q>d$ and choose the function $\eta$ in \eqref{181008@eq5} satisfying
$$
0\le \eta\le 1, \quad \eta \equiv1 \,\text{ on }\, B_{r/2}, \quad \operatorname{supp}\eta\subset B_r, \quad r|\nabla \eta|+r^2|\nabla^2\eta|\le C(d).
$$
Note that
$$
\int_{B_r} h \chi_{\Omega_r}\,dx=\int_{\Omega}h\,dx=0.
$$
Hence, from the existence of solutions to the divergence
equation in a ball, there exist $\hat{h}_\alpha\in W^{1,q}_0(B_r)^d$, $\alpha\in \{1,\ldots,d\}$, such that
\begin{equation}		\label{180302@A1}
D_\alpha \hat{h}_\alpha=h \chi_{\Omega_r} \, \text{ in }\, B_r, \quad \|D \hat{h}_\alpha\|_{L^q(B_r)}\le C(d,q)\|h\|_{L^q(\Omega_r)}.
\end{equation}
We extend $\hat{h}_\alpha$ by zero on $\Omega\setminus B_r$ to see that  $(v, \pi)$ satisfies
$$
\left\{
\begin{aligned}
\cL v+\nabla \pi=D_\alpha (\hat{h}_\alpha+ h_\alpha)&\quad \text{in }\, \Omega,\\
\operatorname{div} v= g &\quad \text{in }\, \Omega,\\
v=0 &\quad \text{on }\, \partial \Omega.
\end{aligned}
\right.
$$
Since the coefficients and data of the above system are of Dini mean oscillation in $\Omega$,
we obtain by \cite[Eq. (2.27)]{MR3890946}  that
\begin{equation}		\label{180214@eq2a}
\begin{aligned}
&\|Dv\|_{L^\infty(\Omega_{r/2})}+\|\pi\|_{L^\infty(\Omega_{r/2})}\\
&\le C r^{-d}\big(\|Dv\|_{L^1(\Omega_r)}+\|\pi\|_{L^1(\Omega_r)}\big) +C\big(\|\hat{h}_\alpha+h_\alpha\|_{L^\infty(\Omega_r)}+\cH(r)\big),
\end{aligned}
\end{equation}
where $C=C(d,\lambda, \gamma, \omega_A,R_0,\varrho_0)$ and
$$
\cH(r)=\int_0^r \frac{\omega^\sharp_{\hat{h}_\alpha+h_\alpha}(t)+\omega^\sharp_g(t)}{t}\,dt.
$$
Here, we use the notation (see \cite[Section 2.1]{MR3890946})
\begin{align*}
&\omega^\sharp_{\bullet}(\rho)=\sup_{\rho\le \rho_0\le R_1} \left(\frac{\rho}{\rho_0}\right)^\gamma \tilde{\omega}_{\bullet}(\rho_0),\\
& \tilde{\omega}_{\bullet} (\rho)=\sum_{i=1}^\infty \kappa^{\gamma i}\big(\omega_{\bullet}(\kappa^{-i}\rho)[\kappa^{-i}\rho<1]+\omega_{\bullet}(1) [\kappa^{-i}\rho\ge 1]\big),\\
&\omega_{f_0}(\rho)=\sup_{x\in \overline{\Omega}}\dashint_{\Omega_\rho(x)}|f_0-(f_0)_{\Omega_\rho(x)}|\,dy \quad \text{for }\, f_0\in L^1(\Omega),
\end{align*}
where $\gamma\in (1-d/q,1)$ is a fixed constant and $\kappa=\kappa(d,\lambda, \gamma, R_0, \varrho)\in (0,1/8]$  is the constant from \cite[Lemma 2.3]{MR3890946}.
We extend $u$ by zero on $B_r\setminus \Omega$. Since
$D\eta u\in W^{1,q}_0(B_r)^d$,
by both Morrey and Poincar\'e inequalities with a scaling, we have
\begin{align*}
r^{-d/q+1}[D\eta u]_{C^{1-d/q}(\overline{\Omega})}+\|D \eta u\|_{L^\infty(\Omega)}&\le C r^{-d/q+1}\|D(D \eta u)\|_{L^q(B_r)}\\
&\le C r^{-d/q}\big( r^{-1} \|u\|_{L^q(B_r)}+\|Du\|_{L^q(\Omega_r)}\big),\\
&\le C \big(r^{-d-1}\|u\|_{L^1(\Omega_r)}+r^{-d/q}\|Du\|_{L^q(\Omega_r)}\big).
\end{align*}
Then for any $0<\rho\le R_1$, we obtain that
$$
\begin{aligned}
\omega_{h_\alpha}(\rho)&\le C \big( \rho^{1-d/q}[D\eta u]_{C^{1-d/q}(\overline{\Omega})}+\omega_A(\rho)\|D\eta u\|_{L^\infty(\Omega)}\big)\\
&\le C\bigg(\frac{\rho^{1-d/q}}{r^{1-d/q}}+\omega_A(\rho)\bigg)\big(r^{-d-1}\|u\|_{L^1(\Omega_r)}+r^{-d/q}\|Du\|_{L^q(\Omega_r)}\big).
\end{aligned}
$$
This together with $1-d/q<\gamma<1$ implies
\begin{equation}		\label{190308@eq1}
\omega_{h_\alpha}^\sharp(\rho)\le C\bigg(\frac{\rho^{1-d/q}}{r^{1-d/q}}+\omega_A^\sharp(\rho)\bigg)\big(r^{-d-1}\|u\|_{L^1(\Omega_r)}+r^{-d/q}\|Du\|_{L^q(\Omega_r)}\big),
\end{equation}
where $C=C(d,\lambda, \gamma, q)$.
Similarly, from the fact that (using \eqref{180302@A1})
\begin{align*}
[\hat{h}_{\alpha}]_{C^{1-d/q}(\overline{\Omega})}&\le C\|D\hat{h}_{\alpha}\|_{L^q(B_r)}\le C \|h\|_{L^q(\Omega_r)}\\
&\le C r^{-1} \big(\|Du\|_{L^q(\Omega_r)}+\|p\|_{L^q(\Omega_r)}\big)+Cr^{d/q} \|f\|_{L^\infty(\Omega_r)},
\end{align*}
we get
$$
\omega^\sharp_{\hat{h}_\alpha}(\rho) \le  C\bigg(\frac{\rho^{1-d/q}}{r^{1-d/q}}\bigg) \big(r^{-d/q} \big(\|Du\|_{L^q(\Omega_r)}+\|p\|_{L^q(\Omega_r)}\big)+r\|f\|_{L^\infty(\Omega_r)}\big).
$$
Combining the above inequality and \eqref{190308@eq1}, we have 
\begin{equation}		\label{180321@A1}
\begin{aligned}
&\cH(r)\le C r^{-d-1}\|u\|_{L^1(\Omega_r)}\\
&\quad+C r^{-d/q}\big(\|Du\|_{L^q(\Omega_r)}+\|p\|_{L^q(\Omega_r)}\big)+C\big(r\|f\|_{L^\infty(\Omega_r)}+|\ell|\big),
\end{aligned}
\end{equation}
where $C=C(d,\lambda,\gamma, q,\omega_A)$.
Therefore, it follows from \eqref{180214@eq2a} and \eqref{180321@A1} that
\begin{equation}		\label{180627@eq1}
\begin{aligned}
&\|Du\|_{L^\infty(\Omega_{r/2}(y))}+\|p\|_{L^\infty(\Omega_{r/2}(y))}\le C r^{-d-1}\|u\|_{L^1(\Omega_r(y))} \\
&\quad +C r^{-d/q} \big(\|Du\|_{L^q(\Omega_r(y))}+\|p\|_{L^q(\Omega_r(y))}\big)+C\big(r \|f\|_{L^\infty(\Omega_r(y))}+|\ell|\big)
\end{aligned}
\end{equation}
for any $y\in \Omega_{R}(x_0)$ and $0<r\le \min\{R_1, \operatorname{dist}(y, \partial B_R(x_0)\}$,
where the constant $C$ depends only on $d$, $\lambda$, $\omega_A$, $R_0$, and $\varrho_0$.

We now complete the proof of \eqref{180304@eq1a}.
Set $U=|Du|+|p|$ and let $R/2\le \rho<r\le R$ with $r-\rho\le R_1$.
Then for any $y\in \Omega_\rho(x_0)$, we obtain by \eqref{180627@eq1} that
$$
\begin{aligned}
&\|U\|_{L^\infty(\Omega_{(r-\rho)/2}(y))}\le C r^{-d-1}\|u\|_{L^1(\Omega_{r-\rho}(y))} \\
&\quad +C \big((r-\rho)^{-d/q} \|U\|_{L^q(\Omega_{r-\rho}(y))}+(r-\rho) \|f\|_{L^\infty(\Omega_{r-\rho}(y))}+|\ell|\big),
\end{aligned}
$$
and thus, we get from Young's inequality that
$$
\begin{aligned}
&\|U\|_{L^\infty(\Omega_{(r-\rho)/2}(y))}\le \delta \|U\|_{L^\infty(\Omega_{r-\rho}(y))}+C r^{-d-1}\|u\|_{L^1(\Omega_{r-\rho}(y))} \\
&\quad +C_\delta (r-\rho)^{-d} \|U\|_{L^1(\Omega_{r-\rho}(y))}+C\big((r-\rho) \|f\|_{L^\infty(\Omega_{r-\rho}(y))}+|\ell|\big)
\end{aligned}
$$
for any $\delta\in (0,1]$, where $C=C(d,\lambda, \omega_A, R_0, \varrho_0)$ and $C_\delta$ depends also on $\delta$.
Since $y$ is an arbitrary point in $\Omega_\rho(x_0)$ and $\Omega_{r-\rho}(y)\subset \Omega_r(x)$, we have
\begin{equation}		\label{180627@eq3}
\begin{aligned}
&\|U\|_{L^\infty(\Omega_{\rho}(x))}\le \delta \|U\|_{L^\infty(\Omega_{r}(x))}+C r^{-d-1}\|u\|_{L^1(\Omega_{r}(x))} \\
&\quad +C_\delta (r-\rho)^{-d} \|U\|_{L^1(\Omega_{r}(x))}+C\big((r-\rho) \|f\|_{L^\infty(\Omega_{r}(x))}+|\ell|\big)
\end{aligned}
\end{equation}
for any $R/2\le \rho < r\le R$ with $r-\rho\le R_1$.
Set
$$
r_k=R\bigg(1-\frac{1}{2^k}\bigg), \quad k\in \{1,2,\ldots\},
$$
and let $k_0$ be the smallest positive integer depending only on $R_1$, such that
$$
r_{k_0+1}-r_{k_0} \le \frac{1}{2^{k_0+1}}\le R_1.
$$
By \eqref{180627@eq3} we have for any $k\in \{k_0, k_0+1,\ldots,\}$,
$$
\begin{aligned}
&\|U\|_{L^\infty(\Omega_{r_k}(x))}
\le \delta \|U\|_{L^\infty(\Omega_{r_{k+1}}(x))}+ \frac{C 2^{dk}}{R^{d+1}}\|u\|_{L^1(\Omega_{r_{k+1}}(x))} \\
&\quad + \frac{C_\delta 2^{dk}}{ R^d} \|U\|_{L^1(\Omega_{r_{k+1}}(x))}+C\big(R \|f\|_{L^\infty(\Omega_{r_{k+1}}(x))}+|\ell|\big).
\end{aligned}
$$
By multiplying both sides of the above inequality by $\delta^k$ and summing the terms with respect to $k\in \{k_0,k_0+1,\ldots\}$, we see that
$$
\begin{aligned}
&\sum_{k=k_0}^\infty\delta^k\|U\|_{L^\infty(\Omega_{r_k}(x))}
\le \sum_{k=k_0+1}^\infty \delta^k \|U\|_{L^\infty(\Omega_{r_{k}}(x))}+ \frac{C}{R^{d+1}}\|u\|_{L^1(\Omega_{R}(x))} \sum_{k=k_0}^\infty (2^d \delta)^k \\
&\quad + \frac{C_\delta}{ R^d} \|U\|_{L^1(\Omega_{R}(x))}\sum_{k=k_0}^\infty (2^d\delta)^k+C\big(R \|f\|_{L^\infty(\Omega_{R}(x))}+|\ell|\big)\sum_{k=k_0}^\infty \delta^k,
\end{aligned}
$$
where each summation is finite upon choosing, for instance, $\delta=2^{-(d+1)}$.
Therefore, by subtracting
$$
\sum_{k=k_0+1}^\infty \delta^k \|U\|_{L^\infty(\Omega_{r_{k}}(x))}
$$
from both sides of the above inequality, we obtain
$$
\begin{aligned}
&\delta^{k_0}\|U\|_{L^\infty(\Omega_{R/2}(x))}
\le  \frac{C}{R^{d+1}}\|u\|_{L^1(\Omega_{R}(x))} \\
&\quad + \frac{C}{ R^d} \|U\|_{L^1(\Omega_{R}(x))}+C\big(R \|f\|_{L^\infty(\Omega_{R}(x))}+|\ell|\big),
\end{aligned}
$$
which implies \eqref{180304@eq1a}.
The lemma is proved.
\end{proof}

The following lemma is analogous to Lemma \ref{180208@lem1}.

\begin{lemma}		\label{180227@lem2}
Let $\Omega=\bR^d_+$, $x_0\in {\Omega}$, and $R\in (0,1]$.
Suppose that the coefficients $A^{\alpha\beta}$ of $\cL$ are of partially Dini mean oscillation with respect to $x'$ in $\bR^d_+$ satisfying Definition \ref{D3} with a Dini function $\omega=\omega_A$.
If $(u,p)\in W^{1,2}(\Omega_R(x_0))^d\times L^2(\Omega_R(x_0))$ satisfies
$$
\left\{
\begin{aligned}
\cL u+\nabla p=f &\quad \text{in }\, \Omega_R(x_0),\\
\operatorname{div} u=0 &\quad \text{in }\, \Omega_R(x_0),\\
u=0 &\quad \text{on }\, B_R(x_0)\cap \partial \Omega,
\end{aligned}
\right.
$$
where $f\in L^\infty(\Omega_R(x_0))^d$, then we have
$$
(u,p)\in W^{1,\infty}(\Omega_{R/2}(x_0))^d\times L^\infty(\Omega_{R/2}(x_0))
$$
with the estimates
\begin{equation}		\label{180327@A1}
\begin{aligned}
&\|Du\|_{L^\infty(\Omega_{R/2}(x_0))}+\|p\|_{L^\infty(\Omega_{R/2}(x_0))}\\
&\le CR^{-d}\big(\|Du\|_{L^1(\Omega_R(x_0))}+\|p\|_{L^1(\Omega_R(x_0))}\big)+CR\|f\|_{L^\infty(\Omega_R(x_0))},
\end{aligned}
\end{equation}
and
$$
\begin{aligned}
&\|u\|_{L^\infty(\Omega_{R/2}(x_0))} \le CR^{-d}\|u\|_{L^1(\Omega_{R/2}(x_0))}\\
&\quad +CR^{1-d}\big(\|Du\|_{L^1(\Omega_R(x_0))}+\|p\|_{L^1(\Omega_R(x_0))}\big)+CR^2\|f\|_{L^\infty(\Omega_R(x_0))},
\end{aligned}
$$
where $C=C(d,\lambda, \omega_A)$.
\end{lemma}

\begin{proof}
With a standard covering argument, we only need to prove the desired estimates with $R/36$ in place of $R/2$ on the left-hand sides.
In this proof, we denote $B_r^+(x)=\Omega_r(x)$ and $B_r^+=B_r^+(0)$ for any $x\in \bR^d$ and  $r>0$.
By the Poincar\'e inequality \eqref{180213@eq2a} with the zero extension of $u$, we only need to prove \eqref{180327@A1}.
Let $z=(z_1,z')\in B_{R/2}^+(x_0)$.
We consider the following two cases:
$$
\operatorname{dist}(z, \partial \bR^d_+)>R/18, \quad \operatorname{dist}(z, \partial \bR^d_+)\le R/18.
$$
\begin{enumerate}[i.]
\item
$\operatorname{dist}(z, \partial \bR^d_+)>R/18$:
In this case, since it holds that
$$
B_{R/18}^+(z)=B_{R/18}(z)\subset B_R^+(x_0),
$$
we get from Lemma \ref{180108@lem3} that
\begin{equation}	\label{180327@A2}
\begin{aligned}
&\|Du\|_{L^\infty(B_{R/36}^+(z))}+\|p\|_{L^\infty(B_{R/36}^+(z))}\\
&\le CR^{-d}\Big(\|Du\|_{L^1(B_R^+(x_0))}+\|p\|_{L^1(B_R^+(x_0))}\Big)+CR\|f\|_{L^\infty(B_R^+(x_0))},
\end{aligned}
\end{equation}
where $C=C(d,\lambda, \omega_A)$.

\item
$\operatorname{dist}(z, \partial \bR^d_+)\le R/18$:
Without loss of generality, we assume that $z'=0'$.
%We use the abbreviation $B_r^+=B_r^+(0)$ for all $r>0$.
To complete the proof of the lemma, it suffices to prove that
\begin{equation}		\label{180228@eq1b}
\begin{aligned}
&\|Du\|_{L^\infty(B_{R/12}^+)}+\|p\|_{L^\infty(B_{R/12}^+)}\\
&\le CR^{-d}\Big(\|Du\|_{L^1(B_{R/4}^+)}+\|p\|_{L^1(B_{R/4}^+)}\Big)+CR\|f\|_{L^\infty(B_{R/4}^+)},
\end{aligned}
\end{equation}
where $C=C(d,\lambda, \omega_A)$,
because \eqref{180327@A1} follows from \eqref{180327@A2},  \eqref{180228@eq1b}, and the fact that
$$
B_{R/36}^+(z)\subset B_{R/12}^+\subset B_{R/4}^+\subset B_{R}^+(x_0).
$$
Let us set $R_1=R/4$ and fix $q>d$.
By the same reasoning as in \eqref{180228@eq2}, $(u, p)$ satisfies
$$
\left\{
\begin{aligned}
\cL u+\nabla p=D_\alpha f_\alpha &\quad \text{in }\, B_{R_1}^+,\\
\operatorname{div} u=0 &\quad \text{in }\, B_{R_1}^+,\\
u=0 &\quad \text{on }\, B_{R_1}\cap \partial \bR^d_+,
\end{aligned}
\right.
$$
where $f_\alpha\in \tilde{W}^{1,q}(B_{R_1}^+)^d$ satisfy
\begin{equation}		\label{180228@eq2a}
\|f_\alpha\|_{L^\infty(B_{R_1}^+)}+R_1^{1-d/q}[f_\alpha]_{C^{1-d/q}(B_{R_1}^+)}\le C(d,q)R_1 \|f\|_{L^\infty(B_{R_1}^+)}.
\end{equation}
We denote
$$
\hat{u}(y)=u(x), \quad \hat{p}(y)=p(x), \quad \hat{A}^{\alpha\beta}(y)=A^{\alpha\beta}(x), \quad \hat{f}_\alpha(y)=f_\alpha (x),
$$
where $y=6x/R_1$, and observe that
\begin{equation}		\label{180228@eq3}
\left\{
\begin{aligned}
D_\alpha(\hat{A}^{\alpha\beta}D_\beta \hat{u})+\nabla (R_1\hat{p}/6)=D_\alpha (R_1\hat{f}_\alpha/6) &\quad \text{in }\, B_6^+,\\
\operatorname{div} \hat{u}=0 &\quad \text{in }\, B_6^+,\\
\hat{u}=0 &\quad \text{on }\, B_6 \cap \partial \bR^d_+.
\end{aligned}
\right.
\end{equation}
Fix a constant $\gamma$ satisfying $1-d/q<\gamma<1$, and let $\kappa=\kappa(d,\lambda,\gamma)\in (0,1/2]$ be the constant from \cite[Lemma 7.1]{MR3912724}.
For $r\in (0,1]$ and $h\in L^1(B_6^+)$, we denote
$$
\begin{aligned}
&\omega_{h,y'}(r)=\sup_{y_0\in B_3^+} \dashint_{B_r^+(y_0)}\bigg| h(y)-\dashint_{B_r'(y_0')}h(y_1,z')\,dz'\bigg|\,dy,\\
&\tilde{\omega}_{h,y'}(r)=\sum_{i=1}^\infty \kappa^{\gamma i}\big(\omega_{h,y'}(\kappa^{-i}r)[\kappa^i r<1]+\omega_{h,y'}(1)[\kappa^{-i}r\ge 1]\big),\\
&\omega^\sharp_{h,y'}(r)=\sup_{r\le r_0\le 1}\bigg(\frac{r}{r_0}\bigg)^{\gamma} \tilde{\omega}_{h,y'}(r_0).
\end{aligned}
$$
By \eqref{180228@eq2a}, we obtain that (using $1-d/q<\gamma<1$)
\begin{equation}		\label{180228@eq5}
R_1\|\hat{f}_1\|_{L^\infty(B_6^+)}+\int_0^1 \frac{\omega^\sharp_{R_1 \hat{f}_\alpha,y'}(t)}{t}\,dt\le CR_1^2 \|f\|_{L^\infty(B^+_{R_1})},
\end{equation}
where $C=C(d,q,\gamma)$.
Since $(\hat{u}, \hat{p})$ satisfies \eqref{180228@eq3}, we obtain by \cite[Eq. (7.6)]{MR3912724} that
$$
\begin{aligned}
&\|D\hat{u}\|_{L^\infty(B_2^+)}+R_1\|\hat{p}\|_{L^\infty(B_2^+)} \\
&\le C\bigg( \|D\hat{u}\|_{L^1(B_6^+)}+R_1\|\hat{p}\|_{L^1(B_6^+)} +R_1\|\hat{f}_1\|_{L^\infty(B_6^+)}+\int_0^1 \frac{\omega^\sharp_{R_1 \hat{f}_\alpha,y'}(t)}{t}\,dt\bigg),
\end{aligned}
$$
where $C=C(d,\lambda,\omega_0, q,\gamma)$ and $\omega_0:(0,1]\to [0,\infty)$ is a function such that
$$
\int_0^r \frac{\omega^\sharp_{\hat{A}^{\alpha\beta},y'}(t)}{t}\,dt \le \int_0^r \frac{\omega_0(t)}{t}\,dt<\infty \quad \text{for all }\, r\in (0,1].
$$
Therefore, by \eqref{180228@eq5} and the change of variables, we get
$$
\begin{aligned}
&\|Du\|_{L^\infty (B_{R_1/3}^+)}+\|p\|_{L^\infty(B_{R_1/3}^+)}\\
&\le  CR_1^{-d}\Big(\|Du\|_{L^1(B_{R_1}^+)}+\|p\|_{L^1(B_{R_1}^+)}\Big) +CR_1 \|f\|_{L^\infty(B_{R_1}^+)},
\end{aligned}
$$
where $C=C(d,\lambda, \omega_0,q,\gamma)=C(d,\lambda, \omega_0)$.

To complete the proof of \eqref{180228@eq1b}, it remains to show that $\omega_0$ can be derived from $\omega_A$.
We set
$$
\begin{aligned}
&\tilde{\omega}_{A}(r)=\sum_{i=1}^\infty \kappa^{\gamma i}\big(\omega_{A}(\kappa^{-i}r)[\kappa^i r<1]+\omega_{A}(1)[\kappa^{-i}r\ge 1]\big),\\
&\omega^\sharp_{A}(r)=\sup_{r\le r_0\le 1}\bigg(\frac{r}{r_0}\bigg)^{\gamma} \tilde{\omega}_{A}(r_0),
\end{aligned}
$$
and observe that (see \eqref{180228@eq6})
$$
\tilde{\omega}_{\hat{A}^{\alpha\beta},y'}(r)\le \tilde{\omega}_A(rR_1/6)+C_0r^{\gamma} \int_0^1 \frac{\omega_A(t)}{t}\,dt,
$$
where $C=C(d,\lambda, \omega_A)$.
From the above inequality it follows that
\begin{align*}
\omega^\sharp_{\hat{A}^{\alpha\beta},y'}(r)\le \omega^\sharp_A(rR_1/6)+C_0 r^{\gamma}\int_0^1 \frac{\omega_A(t)}{t}\,dt.
\end{align*}
Therefore, for any $r\in (0,1]$, we have
\begin{align*}
\int_0^r \frac{\omega^\sharp_{\hat{A}^{\alpha\beta},y'}(t)}{t}\,dt
&\le \int_0^r \frac{\omega^\sharp_A(tR_1/6)}{t}\,dt+C_0\int_0^r \frac{t^{\gamma}}{t}\,dt\cdot \int_0^1 \frac{\omega_A(t)}{t}\,dt\\
&\le \int_0^r \frac{\omega_0(t)}{t}\,dt<\infty,
\end{align*}
where we set
$$
\omega_0(r)=\omega_A^\sharp(r)+C_0r^{\gamma}\int_0^1 \frac{\omega_A(t)}{t}\,dt.
$$
\end{enumerate}
The lemma is proved.
\end{proof}

\bibliographystyle{plain}

\end{document}